\documentclass[11pt]{amsart}
\usepackage{amsmath,amssymb,amsthm}
\usepackage[dvipsnames]{xcolor}
\usepackage{tikz}
\usetikzlibrary{arrows.meta,calc,decorations.pathmorphing,positioning}
\usepackage[
  colorlinks=true,
  linkcolor=black,
  citecolor=RoyalBlue,
  urlcolor=black,
  filecolor=black,
  linktoc=all,
  breaklinks=true
]{hyperref}
\hypersetup{pdfborder={0 0 0}}
\newtheorem{theorem}{Theorem}[section]
\newtheorem{proposition}[theorem]{Proposition}
\newtheorem{lemma}[theorem]{Lemma}

\newtheorem{remark}[theorem]{Remark}

\newcommand{\CP}{\mathbb{CP}}
\newcommand{\Z}{\mathbb Z}
\newcommand{\Mod}{\operatorname{Mod}}
\newcommand{\Disc}{\operatorname{Disc}}

\definecolor{lineblue}{RGB}{20,82,180}
\definecolor{linered}{RGB}{205,38,54}
\definecolor{linegreen}{RGB}{0,132,92}
\definecolor{lineorange}{RGB}{255,125,0}
\definecolor{linepurple}{RGB}{210,35,190}
\definecolor{lineteal}{RGB}{0,175,205}
\definecolor{vertexfill}{RGB}{255,191,0}
\definecolor{diagfill}{RGB}{236,92,126}
\definecolor{pointedge}{RGB}{32,36,45}
\definecolor{quartic}{RGB}{185,55,115}
\definecolor{fibercol}{RGB}{55,105,180}
\title{Xiao's Genus-Two Fibration: Branched Covers and Braid Monodromy}
\author{Anar Akhmedov}
\address{School of Mathematics, University of Minnesota, Minneapolis, MN 55455, USA}
\address{Department of Mathematics, Harvard University, Cambridge, MA 02138, USA}
\email{akhmedov@umn.edu}
\email{akhmedov@math.harvard.edu}
\date{}
\begin{document}
\begin{abstract}
We compute the geometric monodromy factorization of Xiao's genus-two Lefschetz fibration directly from its branched-cover construction, following Moishezon's braid-monodromy method.  The complete quadrangle determines the motion of the six branch points and gives four nonseparating and three separating vanishing cycles.  At each of the three $3+3$ degenerations, the branch motion determines the spherical mapping class, and Xiao's local holomorphic model determines its genus-two lift.  After fixing a distinguished system of paths, we obtain an ordered positive factorization of type $(4,3)$ and give Artin coordinates for the seven factors.
\end{abstract}
\maketitle

\section{Introduction}
In his 1985 French monograph \emph{Surfaces fibr\'ees en courbes de genre deux}, Xiao described the genus-two fibration with seven singular fibers that is the subject of this paper.  His construction has a particularly concrete algebraic form: it starts with a complete quadrangle and a nodal quartic, blows up seven points of $\mathbb P^2$, and then takes a double cover \cite{Xiao}.  In work first posted on arXiv on September~6, 2015 (arXiv:1509.01853) and later published in the \emph{Kyoto Journal of Mathematics}, Akhmedov--Monden used Matsumoto's and Xiao's genus-two fibrations as building blocks for genus-two Lefschetz fibrations with small topology and for applications to small exotic $4$-manifolds \cite{AM}.  After recalling Xiao's branched-cover construction, they observed that its explicit monodromy should be accessible directly from this description by Moishezon's braid-monodromy techniques \cite[Example~8 and Remark~17]{AM}.  Our purpose here is to carry out that calculation.

We work from Xiao's algebraic branch model, without beginning with a known positive factorization in the genus-two mapping class group.  The calculation follows Moishezon's braid-monodromy method: project the branch curve, choose paths to the critical values, compute the local braids, and lift them to the covering surface; see \cite{MoishezonStable,MoishezonBraidsII,MT1}.  Auroux--Katzarkov developed this viewpoint for symplectic branched coverings of $\mathbb{CP}^2$ and related braid monodromy to the mapping-class-group monodromy of Lefschetz pencils \cite{AKBranched,AKDoubling}; see also \cite{KatzMonodromy}.  For Xiao's example, every step can be carried out explicitly.

We first choose coordinates for the complete quadrangle and an explicit quartic in the corresponding linear system.  After blowing up the seven distinguished points, the branch divisor is smooth, and the pencil through one of the triple points gives a six-point family on $\mathbb P^1$.  The discriminant of the resulting cubic singles out four simple ramification values.  The remaining critical values are $0,1,\infty$, where the six branch points separate into two triples.  Thus the local braid picture can be read directly from the branch geometry.  The four simple tangencies give the nonseparating vanishing cycles.  At the three $3+3$ values the branch motion gives the spherical class $T_\delta^2$; the local holomorphic Lefschetz model in Xiao's relatively minimal fibration then determines the preferred genus-two lift, as explained in Proposition~\ref{prop:sep}.

There is one point at the $3+3$ degenerations that deserves some care.  Lifting the natural disk braid generator by generator does not necessarily give the Lefschetz monodromy: the two genus-two lifts differ by the hyperelliptic involution.  The local Picard--Lefschetz model tells us which lift occurs.  The quartic, branch divisor, discriminant, and critical values are computed exactly over $\mathbb Q$, and the local monodromies come from the corresponding branch degenerations and the chosen paths.  After these geometric monodromies are fixed, we choose a disk chart and continuation paths and numerically follow the roots to obtain explicit Artin representatives.  The numerical calculation is used only to write these representatives explicitly.

We also identify the covering line bundle and the horizontal and vertical ramification curves.  For later use in fiber sums, we recall the fundamental-group facts proved in Akhmedov--Monden \cite{AM}.  That paper also uses Xiao's fibration, together with Matsumoto's genus-two fibration, to construct small exotic $4$-manifolds.  We use the double-cover formulas in \cite{BHPV,Pardini} and Endo's hyperelliptic signature formula \cite{Endo}.

\begin{theorem}[Main theorem]\label{thm:main}
The geometric monodromy of Xiao's genus-two Lefschetz fibration consists of four nonseparating and three separating positive Dehn twists and hence gives a positive genus-two monodromy factorization of type $(4,3)$.  The four nonseparating factors and the three spherical $3+3$ classes are derived from the complete-quadrangle branch model; the preferred separating lifts are fixed by Xiao's local holomorphic Lefschetz model.  Section~\ref{sec:ordered} gives an explicit ordered geometric realization for a fixed distinguished system of paths, together with numerically computed Artin coordinates.
\end{theorem}

\section{The complete quadrangle}
Choose homogeneous coordinates $[x:y:z]$ and put
\[
P_0=[1:0:0],\quad P_1=[0:1:0],\quad P_2=[0:0:1],\quad P_3=[1:1:1].
\]
Take the six lines
\[
\begin{array}{lll}
L_1:y-x=0, & L_2:x-z=0, & L_3:x=0,\\
L'_1:z=0, & L'_2:y=0, & L'_3:y-z=0.
\end{array}
\]
Their three diagonal double points are
\[
P_4=[1:1:0],\qquad P_5=[1:0:1],\qquad P_6=[0:1:1].
\]
This is the complete-quadrangle configuration used in the description of Xiao's fibration recalled by Akhmedov--Monden \cite[Example~8]{AM}.  Figure~\ref{fig:quadrangle} is an original colored schematic of the same incidence pattern.

\begin{figure}[ht]
\centering
\begin{tikzpicture}[scale=1.05,
  every node/.style={font=\small},
  linelabel/.style={fill=white,fill opacity=.92,text opacity=1,inner sep=1.1pt},
  pointlabel/.style={fill=white,fill opacity=.92,text opacity=1,inner sep=1.0pt,font=\small}]
\coordinate (P1) at (-3,-2);
\coordinate (P2) at ( 3,-2);
\coordinate (P3) at ( 0, 3);
\coordinate (P0) at ( 0, 0);
\coordinate (P4) at ( 1.285714,0.857143); 
\coordinate (P5) at (-1.285714,0.857143); 
\coordinate (P6) at ( 0,-2);

\draw[line width=1.85pt,lineblue] (-3.34,-2.57)--(0.34,3.57)
  node[linelabel,pos=.57,left=2pt,text=lineblue] {$L_2$};
\draw[line width=1.85pt,linered]  (3.34,-2.57)--(-0.34,3.57)
  node[linelabel,pos=.57,right=2pt,text=linered] {$L_1$};
\draw[line width=1.85pt,linegreen] (-3.48,-2)--(3.48,-2)
  node[linelabel,pos=.33,below=2pt,text=linegreen] {$L_3$};

\draw[line width=1.85pt,linepurple] (-3.35,-2.233)--(1.82,1.213)
  node[linelabel,pos=.78,above=2pt,text=linepurple] {$L'_1$};
\draw[line width=1.85pt,lineorange] (3.35,-2.233)--(-1.82,1.213)
  node[linelabel,pos=.78,above=2pt,text=lineorange] {$L'_2$};
\draw[line width=1.85pt,lineteal] (0,-2.47)--(0,3.52)
  node[linelabel,pos=.83,right=2pt,text=lineteal] {$L'_3$};

\foreach \p in {P0,P1,P2,P3,P4,P5,P6}{
  \filldraw[fill=black,draw=white,line width=.55pt] (\p) circle (2.75pt);
}

\node[pointlabel,right=4pt]       at (P0) {$P_0$};
\node[pointlabel,below left=2pt]  at (P1) {$P_1$};
\node[pointlabel,below right=2pt] at (P2) {$P_2$};
\node[pointlabel,right=3pt]       at (P3) {$P_3$};
\node[pointlabel,right=4pt]       at (P4) {$P_4$};
\node[pointlabel,left=4pt]        at (P5) {$P_5$};
\node[pointlabel,right=4pt]       at (P6) {$P_6$};
\end{tikzpicture}
\caption{The complete quadrangle, drawn with the incidence pattern used in Akhmedov--Monden \cite[Figure~4 and Example~8]{AM}.  The four triple points are $P_0,P_1,P_2,P_3$ and the three double points are $P_4,P_5,P_6$.}
\label{fig:quadrangle}
\end{figure}
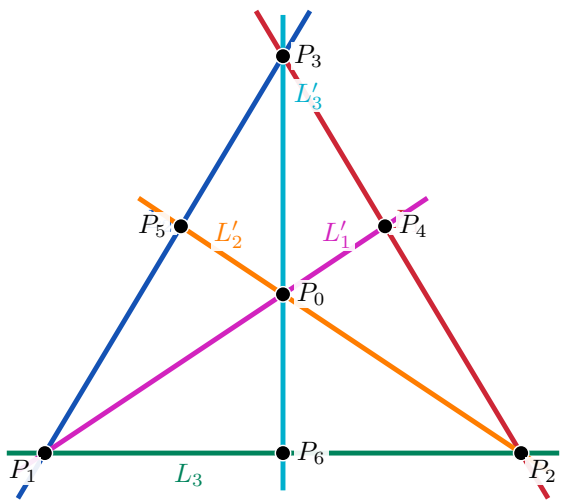

Let $Y=\CP^2\#7\overline{\CP}^{\,2}$ be the blowup at $P_0,\dots,P_6$, with exceptional curves $E_i$ and hyperplane class $H$.  The strict transforms have classes
\[
\begin{aligned}
\widetilde L_1&=H-E_2-E_3-E_4,&
\widetilde L'_1&=H-E_0-E_1-E_4,\\
\widetilde L_2&=H-E_1-E_3-E_5,&
\widetilde L'_2&=H-E_0-E_2-E_5,\\
\widetilde L_3&=H-E_1-E_2-E_6,&
\widetilde L'_3&=H-E_0-E_3-E_6.
\end{aligned}
\]

\section{An explicit quartic and the smooth branch divisor}
The quartic below comes from the linear system determined by the complete quadrangle.  Let
\[
\mathcal V=\left\{F\in H^0(\mathbb P^2,\mathcal O_{\mathbb P^2}(4)):\;
F(P_i)=0\ (0\le i\le3),\ \operatorname{mult}_{P_j}F\ge2\ (4\le j\le6)\right\}.
\]
A homogeneous quartic has $15$ coefficients.  The four simple point conditions and the first-derivative conditions at $P_4,P_5,P_6$ define $13$ linear functionals (Euler's identity makes the value condition at a prescribed singular point redundant once all first derivatives vanish).  With respect to the standard monomial basis of $H^0(\mathbf P^2,\mathcal O_{\mathbf P^2}(4))$, the corresponding $13\times15$ evaluation matrix has rank $13$.  Hence
\[
\dim \mathcal V=2.
\]
Direct row reduction gives the basis
\[
F_1=z(x-y)(x+y-z)^2,\qquad
F_2=y(x-z)(x-y+z)^2.
\]
Thus the relevant quartics form the pencil
\begin{equation}\label{eq:quarticpencil}
C_{[\lambda:\mu]}:\quad \lambda F_1+\mu F_2=0,
\qquad [\lambda:\mu]\in\mathbb P^1.
\end{equation}
We choose the particularly symmetric member $[\lambda:\mu]=[1:1]$:
\begin{equation}\label{eq:quartic}
 C:\quad z(x-y)(x+y-z)^2+y(x-z)(x-y+z)^2=0.
\end{equation}
Proposition~\ref{prop:quartic} shows that this member is irreducible and has exactly the three prescribed ordinary nodes.  It follows that the members with no additional singularities form a nonempty Zariski-open subset of the pencil, containing the member chosen here.
Expanding gives
\[
\begin{aligned}
C={}&x^3y+x^3z-2x^2y^2+2x^2yz-2x^2z^2+xy^3-xy^2z\\
&-xyz^2+xz^3-2y^3z+4y^2z^2-2yz^3.
\end{aligned}
\]

\begin{proposition}\label{prop:quartic}
The quartic $C$ is absolutely irreducible.  Its singular locus is exactly
$\{P_4,P_5,P_6\}$, and all three singularities are ordinary nodes.  Its strict transform is a smooth rational curve of class
\[
\widetilde C=4H-E_0-E_1-E_2-E_3-2E_4-2E_5-2E_6.
\]
\end{proposition}
\begin{proof}
We include the calculation because both assertions will be used later.  Direct substitution shows first that $C$ passes through $P_0,\dots,P_6$.

Put $f(y,z)=C(1,y,z)$.  Over $\mathbb Q$, a direct elimination of the affine singularity equations gives the reduced Gr\"obner basis
\[
 \langle f,f_y,f_z\rangle
 =\langle y+z-1,\ z^2-z\rangle
\]
for lexicographic order $y>z$.  Hence the only singular points with $x\neq0$ are
\[
 [1:1:0]=P_4,\qquad [1:0:1]=P_5.
\]
It remains to inspect the line $x=0$.  There
\[
\begin{aligned}
 C(0,y,z)&=-2yz(y-z)^2,\\
 C_x(0,y,z)&=(y-z)^2(y+z),\\
 C_y(0,y,z)&=-2z(y-z)(3y-z),\\
 C_z(0,y,z)&=-2y(y-z)(y-3z).
\end{aligned}
\]
If $y=z\neq0$ we obtain $P_6=[0:1:1]$.  If $y\neq z$, the equation $C_x=0$ forces $y+z=0$; then $C_y\neq0$ for a projective point.  Thus there are no further singularities, and
\[
\operatorname{Sing}(C)=\{P_4,P_5,P_6\}.
\]

The quadratic tangent cones are also explicit.  At $P_4$ and $P_5$ they are, after suitable affine coordinates $(A,B)$,
\[
 A^2-6AB+B^2,
\]
and at $P_6$ the tangent cone is
\[
 -2(A^2+B^2).
\]
Each has two distinct complex linear factors.  Therefore all three singularities are ordinary nodes.

We now prove absolute irreducibility without appealing to a computer factorization.  Since the singular locus is finite, $C$ is reduced.  Suppose, over $\mathbb C$, that $C=\bigcup_{i=1}^r C_i$ is its decomposition into irreducible components.  By B\'ezout, any two positive-degree projective plane curves have nonempty intersection, so the support of $C$ is connected.  Since every intersection between distinct components would occur among the three ordinary nodes already identified, the total $\delta$-invariant is $\delta=3$.  For a connected nodal curve the normalization formula gives
\[
 p_a(C)=\sum_{i=1}^r g(C_i)+\delta-r+1.
\]
Here $p_a(C)=3$ and $\delta=3$, so
\[
 \sum_{i=1}^r g(C_i)=r-1.
\]
If $r\ge3$, the degree partition of $4$ forces every component to have degree at most two, hence the left side is zero, a contradiction.  Thus $r=2$.  A $2+2$ decomposition again has zero total geometric genus and is impossible.  The only remaining possibility is a line plus a cubic, and the cubic must have genus one, hence be smooth.  Consequently the three nodes must be precisely the three transverse line--cubic intersections.  The line component would therefore contain all three nodes $P_4,P_5,P_6$.  But
\[
\det\!\begin{pmatrix}1&1&0\\1&0&1\\0&1&1\end{pmatrix}=-2\neq0,
\]
so these three points are not collinear.  Hence no line--cubic decomposition exists and $C$ is absolutely irreducible.

The divisor class follows from the four simple base points $P_0,\dots,P_3$ and the three double points $P_4,P_5,P_6$.  Finally $p_a(C)=3$ and the three nodes have $\delta=1$, so the normalization has genus zero.  Blowing up the three nodes resolves them, while blowing up the smooth points $P_0,\dots,P_3$ preserves smoothness of the strict transform; therefore $\widetilde C$ is smooth and rational.
\end{proof}

\begin{remark}\label{rem:branchsmooth}
The disjointness of the branch components after blowup can be checked without a genericity assertion.  Restricting \eqref{eq:quartic} to the six sides gives
\[
\begin{array}{c|c}
\text{side} & C|_{L}\\ \hline
L_1:\ y=x & xz^2(x-z)\\
L_2:\ z=x & xy^2(x-y)\\
L_3:\ x=0 & -2yz(y-z)^2\\
L'_1:\ z=0 & xy(x-y)^2\\
L'_2:\ y=0 & xz(x-z)^2\\
L'_3:\ y=z & 2x^2z(x-z).
\end{array}
\]
Thus every line--quartic intersection is one of the seven prescribed points.  At a diagonal node the corresponding factor has multiplicity exactly two, so the line is tangent to neither local branch of the node: tangency to either branch would give local intersection multiplicity at least three, since the other branch also passes through the node.  After one blowup the strict transform of the line is disjoint from the two points where the normalization of $C$ meets the exceptional curve.  At the four triple points the relevant intersections are simple, and the blowup separates the distinct tangent directions.  The six line components are themselves separated by the seven blowups.  Hence the seven strict transforms forming $D$ are pairwise disjoint.
\end{remark}

The branch divisor is
\[
D=\widetilde L_1+\widetilde L_2+\widetilde L_3+
\widetilde L'_1+\widetilde L'_2+\widetilde L'_3+\widetilde C
=10H-4(E_0+\cdots+E_6).
\]
By Remark~\ref{rem:branchsmooth}, the seven components are pairwise disjoint on $Y$; equivalently, for each of the six sides the prescribed simple and nodal intersections exhaust its intersection number four with $C$.  The blowups also separate all intersections among the six line components.  Hence $D$ is smooth and divisible by two in $\operatorname{Pic}(Y)$.

Let
\[
p:\widetilde S\longrightarrow Y
\]
be the double cover branched along $D$.  Figure~\ref{fig:cover} shows the two maps in the construction.

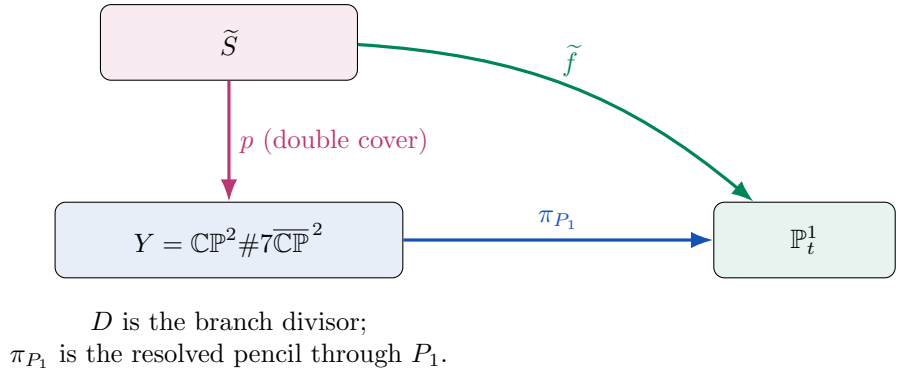
\begin{figure}[ht]
\centering
\begin{tikzpicture}[>=Latex, node distance=2.0cm, every node/.style={font=\small}]
\node[draw,rounded corners,fill=quartic!10,minimum width=3.4cm,minimum height=1.0cm] (S) {$\widetilde S$};
\node[draw,rounded corners,fill=lineblue!10,below=1.6cm of S,minimum width=4.6cm,minimum height=1.0cm] (Y) {$Y=\CP^2\#7\overline{\CP}^{\,2}$};
\node[draw,rounded corners,fill=linegreen!10,right=4.1cm of Y,minimum width=2.4cm,minimum height=1.0cm] (B) {$\mathbb P^1_t$};
\draw[->,very thick,quartic] (S)--node[right] {$p$ (double cover)} (Y);
\draw[->,very thick,lineblue] (Y)--node[above] {$\pi_{P_1}$} (B);
\draw[->,very thick,linegreen,bend left=18] (S) to node[above] {$\widetilde f$} (B);
\node[below=0.3cm of Y,align=center] {$D$ is the branch divisor;\\$\pi_{P_1}$ is the resolved pencil through $P_1$.};
\end{tikzpicture}
\caption{The branch-cover picture.  Pulling back the resolved pencil through $P_1$ gives a genus-two fibration on $\widetilde S$; contracting the three vertical $(-1)$-curves gives Xiao's relatively minimal fibration.}
\label{fig:cover}
\end{figure}

\section{The double cover in detail}\label{sec:doublecover}
We use the branch-cover model recalled in \cite[Example~8]{AM}, keeping track of the covering line bundle, the ramification curves, and their intersections with the pencil.

Put
\[
D=\widetilde L_1+\widetilde L_2+\widetilde L_3+
  \widetilde L'_1+\widetilde L'_2+\widetilde L'_3+\widetilde C.
\]
Using the divisor classes above,
\begin{equation}\label{eq:Dclass}
 D=10H-4(E_0+\cdots+E_6)=2\mathcal L,
 \qquad
 \mathcal L=5H-2(E_0+\cdots+E_6).
\end{equation}
Since the seven components of $D$ are smooth and pairwise disjoint, $D$ is a smooth, reduced (not necessarily connected) divisor.  The line bundle $\mathcal O_Y(\mathcal L)$ and a section with zero divisor $D$ determine the standard degree-two cover \cite{Pardini,BHPV}
\[
p:\widetilde S\longrightarrow Y.
\]
Because $D$ is reduced and nonempty, the defining section is not a square in the function field of $Y$; hence the cover is connected.  Indeed, along every component of the reduced branch divisor $D$, the defining section has vanishing order one, whereas the divisor of a square has even valuation along every prime divisor.  The local equation $w^2=s$ along $D$ also shows that $\widetilde S$ is smooth.

\begin{proposition}[Algebraic construction of the fibration]\label{prop:algconstruction}
Let $\pi_{P_1}:Y\to\mathbb P^1$ be the resolved pencil through $P_1$.  The composite $\widetilde f=\pi_{P_1}\circ p$ has connected generic fiber of genus two.  Three ramification curves are sections of $\widetilde f$, three are vertical $(-1)$-curves, and the ramification curve over $\widetilde C$ is a degree-three multisection.  Contracting the vertical $(-1)$-curves gives the relatively minimal genus-two fibration $f:S\to\mathbb P^1$ used below.
\end{proposition}
\begin{proof}
The intersection table below gives $D\cdot F=6$ and separates the six side components into three horizontal and three vertical ones.  Thus a generic fiber is a connected double cover of $\mathbb P^1$ branched at six points, and Riemann--Hurwitz gives genus two.  Formula \eqref{eq:ramself} shows that the ramification lift of every side has self-intersection $-1$.  The three horizontal lifts meet a generic fiber once and are sections; the three vertical lifts lie in special fibers.  The ramification lift of $\widetilde C$ maps with degree $\widetilde C\cdot F=3$ to the base.  Contracting the three vertical lifts gives the relative minimalization described for this branch model in \cite[Example~8]{AM}.
\end{proof}
Locally near a point of $D$ the cover has the form
\[
w^2=s,
\]
where $s=0$ is a local equation of $D$.  Thus the ramification divisor $R\subset\widetilde S$ maps isomorphically to $D$ and satisfies
\[
p^*D=2R.
\]
For an individual branch component $B\subset D$ with ramification lift $R_B$, one has
\begin{equation}\label{eq:ramself}
R_B^2=\frac12 B^2.
\end{equation}
Indeed, $(p^*B)^2=2B^2$ and $p^*B=2R_B$, so $4R_B^2=2B^2$.
Every strict transform of one of the six sides has square $-2$, hence its ramification lift is a $(-1)$-sphere.

For a smooth double cover the canonical bundle formula gives \cite{Pardini,BHPV}:
\begin{equation}\label{eq:canonicalcover}
K_{\widetilde S}=p^*(K_Y+\mathcal L)
=p^*\bigl(2H-E_0-\cdots-E_6\bigr).
\end{equation}
Consequently
\[
K_{\widetilde S}^2
=2\bigl(2H-E_0-\cdots-E_6\bigr)^2=-6,
\]
and the double-cover holomorphic Euler-characteristic formula \cite{Pardini,BHPV} gives
\[
\chi(\mathcal O_{\widetilde S})
=2\chi(\mathcal O_Y)+\frac12\mathcal L(\mathcal L+K_Y)
=2+\frac12(10-14)=0.
\]
These calculations agree with the invariants recorded for the same branched cover in \cite[Example~8]{AM}.

Let
\[
F=H-E_1
\]
be the fiber class of the resolved pencil of lines through $P_1$.  Then
\[
F^2=0,
\qquad
D\cdot F=6.
\]
The six branch points on a generic pencil member therefore arise before any local coordinate calculation.  More precisely,
\[
\begin{array}{c|c}
\text{branch component}&F\cdot B\\ \hline
\widetilde L'_1,\widetilde L_2,\widetilde L_3&0\\
\widetilde L_1,\widetilde L'_2,\widetilde L'_3&1\\
\widetilde C&3.
\end{array}
\]
Thus three line components are vertical components of special members of the pencil, three are genuine branch sections, and the quartic contributes the remaining three moving branch points.  The divisor calculation therefore gives the decomposition
\[
6=1+1+1+3
\]
that appears explicitly in the hyperelliptic equation below.

The three horizontal branch components have square $-2$, and by \eqref{eq:ramself} their ramification lifts are $(-1)$-curves meeting a generic genus-two fiber once.  They are therefore three disjoint $(-1)$-sections.  The three vertical branch-line components also lift to $(-1)$-curves, but these lie in singular fibers and are exactly the curves contracted in passing to the relatively minimal Xiao fibration.  This section/vertical-component dichotomy is the one described geometrically in \cite[Example~8]{AM}.

Finally, the preimage $\widetilde F=p^{-1}(F)$ of a generic pencil member is a double cover of $F\cong\mathbb P^1$ branched at six points.  Riemann--Hurwitz gives
\[
2g(\widetilde F)-2=2(-2)+6=2,
\]
so $g(\widetilde F)=2$.  Thus the branch divisor class alone already forces genus two.

\section{The pencil and the six hyperelliptic branch points}\label{sec:pencil}
The pencil through $P_1$ is
\[
\ell_t:\ z=tx,\qquad t\in\mathbb P^1,
\]
with $\ell_\infty=\{x=0\}$ and fiber class $F=H-E_1$.  Since
\[
D\cdot F=(10H-4\sum E_i)(H-E_1)=6,
\]
a generic pencil member meets $D$ in six points.  Its inverse image is therefore a genus-two double cover of $\mathbb P^1$.

Set $x=1$ and $u=y/x$.  The three line components not containing $P_1$ contribute
\[
u=0,\qquad u=1,\qquad u=t.
\]
Substitution in \eqref{eq:quartic} gives
\begin{equation}\label{eq:q}
q_t(u)=t(1-u)(1+u-t)^2+u(1-t)(1-u+t)^2.
\end{equation}
The other three branch points are the roots $r_1(t),r_2(t),r_3(t)$ of $q_t$, viewed on the projective fiber $\mathbb P^1_u$.  When $t=\frac12$ the cubic term vanishes and one root moves to $u=\infty$; that projective root is simple, so $t=\frac12$ is not a critical value.  Indeed, the homogenized cubic at $t=\frac12$ is $V(-\frac32U^2+\frac32UV+\frac18V^2)$, so $[U:V]=[1:0]$ is a simple root.  A regular fiber is therefore birational to
\begin{equation}\label{eq:family}
v^2=u(u-1)(u-t)q_t(u).
\end{equation}
Collisions of a quartic branch point with one of the three distinguished branch sections are read off from
\begin{equation}\label{eq:sectioncollisions}
q_t(0)=t(1-t)^2,\qquad q_t(1)=t^2(1-t),\qquad q_t(t)=-2t(t-1).
\end{equation}
Thus, for finite $t$, a root of $q_t$ can meet one of $0,t,1$ only at $t=0$ or $t=1$.  The fiber over $t=\infty$ is treated in the $s=1/t$ chart in Section~7.  The discriminant calculation below shows that the only other collisions are the four simple ramification values of the trigonal curve.  Hence the six branch points are distinct over the complement of these seven values.

\begin{figure}[ht]
\centering
\begin{tikzpicture}[scale=1.05,every node/.style={font=\small}]
\draw[->,thick] (-0.4,0)--(9.4,0) node[right] {$u$};
\foreach \x/\lab/\col in {0.6/$0$/lineblue,2.1/$t$/linegreen,3.6/$1$/linered,5.1/$r_1(t)$/quartic,6.7/$r_2(t)$/quartic,8.3/$r_3(t)$/quartic}{\fill[\col] (\x,0) circle (3.2pt);\node[above=4pt] at (\x,0) {\lab};}
\draw[line width=.8pt] (0.55,-.35)--(3.65,-.35); \node at (2.1,-.65) {three line components};
\draw[line width=.8pt] (5.05,-.35)--(8.35,-.35); \node at (6.7,-.65) {three points on $\widetilde C$};
\end{tikzpicture}
\caption{The six branch points on a generic member of the pencil: three are supplied by line components and three move on the normalized quartic.  Their actual order varies with $t$; the picture is schematic.}
\label{fig:sixpoints}
\end{figure}
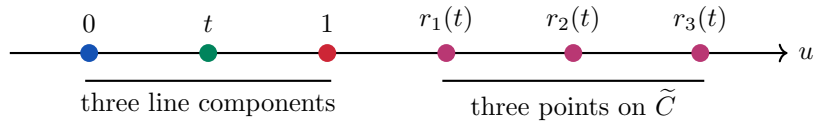

\section{The discriminant and the four tangencies}
An exact discriminant calculation in $\mathbb Q[t]$ gives
\begin{equation}\label{eq:disc}
\Disc_u(q_t)=-32t^2(t-1)^2R(t),
\end{equation}
where
\begin{equation}\label{eq:R}
R(t)=2t^4-4t^3+12t^2-10t-1.
\end{equation}
Moreover
\[
\gcd(R,R')=1,\qquad \operatorname{Disc}_t(R)=-1259712\neq0,
\]
so the four roots of $R$ are simple.  Since $R(\frac12)=-\frac{27}{8}$, none of them is the value at which the leading coefficient $1-2t$ of $q_t$ vanishes.  Thus $q_\tau$ is a genuine cubic at every root $\tau$ of $R$.  The factors $t^2$ and $(t-1)^2$ record the nodal directions in the plane model; the third analogous nodal direction occurs at $t=\infty$.  They are not simple ramification values of the normalized trigonal curve.  The exact algebra behind these statements is recorded in Appendix~\ref{app:exact}.

Since $\widetilde C\simeq\mathbb P^1$ and $\widetilde C\cdot F=3$, projection from $P_1$ restricts to a degree-three map
\[
\rho:\widetilde C\longrightarrow\mathbb P^1_t.
\]
Riemann--Hurwitz gives
\[
-2=3(-2)+\deg R_\rho,
\]
so $\deg R_\rho=4$.  Since the four roots of $R$ are simple and account for the full ramification degree, they are precisely the four simple ramification values of $\rho$.

Writing $s=t-\frac12$ gives
\[
R(t)=2s^4+9s^2-\frac{27}{8},
\]
so the four values are
\[
\tau_{L,R}=\frac12\mp\frac12\sqrt{-9+6\sqrt3},\qquad
\tau_{U,D}=\frac12\pm\frac{i}{2}\sqrt{9+6\sqrt3}.
\]

\begin{figure}[ht]
\centering
\begin{tikzpicture}[x=2.25cm,y=1.05cm,every node/.style={font=\small}]
\draw[->,thick] (-.72,0)--(1.72,0) node[right] {$\Re t$};
\draw[->,thick] (0,-2.62)--(0,2.62) node[above] {$\Im t$};
\draw[densely dashed,thin] (.5,-2.42)--(.5,2.42);
\fill[lineblue] (-.090,0) circle (3pt) node[below left=3pt] {$\tau_L$};
\fill[lineblue] (1.090,0) circle (3pt) node[below right=3pt] {$\tau_R$};
\fill[quartic] (.5,2.202) circle (3pt) node[right=5pt] {$\tau_U$};
\fill[quartic] (.5,-2.202) circle (3pt) node[right=5pt] {$\tau_D$};
\fill[linegreen] (0,0) circle (3pt) node[above right=3pt] {$0$};
\fill[linered] (1,0) circle (3pt) node[above left=3pt] {$1$};
\draw[->,lineorange,thick] (1.38,1.32)--(1.62,1.90);
\node[lineorange] at (1.42,2.05) {$\infty$};
\end{tikzpicture}
\caption{The seven critical values in the base: four tangency values (blue/magenta) and the three separating values $0,1,\infty$ (green/red/orange).}
\label{fig:critical}
\end{figure}
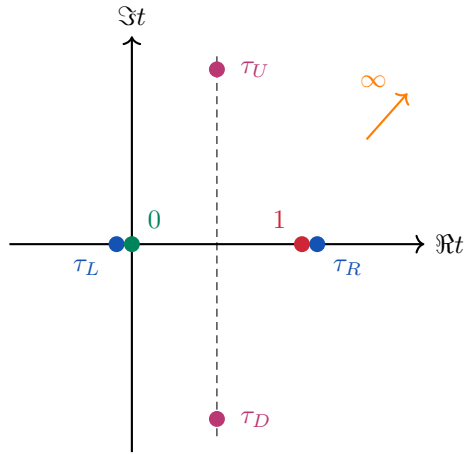

\begin{proposition}\label{prop:half}
For each $\tau\in\{\tau_L,\tau_R,\tau_U,\tau_D\}$, exactly two roots of $q_t$ coalesce.  A positive loop about $\tau$ gives a positive half twist $h_{\alpha_\tau}$ about an arc $\alpha_\tau$ joining those two branch points.  Its hyperelliptic lift is a positive Dehn twist $t_{a_\tau}$ about the nonseparating curve $a_\tau=p^{-1}(\alpha_\tau)$.
\end{proposition}
\begin{proof}
A simple zero of the discriminant gives local roots
\[
u_\pm(t)=u_0\pm c\sqrt{t-\tau}+O(t-\tau),\qquad c\neq0.
\]
A positive circuit exchanges the two roots by a positive half twist.  The local double-cover model also shows directly that this is a positive Lefschetz singularity.  Put $s=t-\tau$ and $x=u-u_0$.  Since the ramification of $\rho$ is simple, after multiplication by a holomorphic unit the two colliding branch factors have the form $x^2-\kappa s+O(x^3,xs,s^2)$ for some nonzero constant $\kappa\in\mathbb C^\times$; the other four branch factors are nonzero at the collision.  Absorbing the resulting unit into the covering coordinate, the double-cover equation has a nondegenerate $A_1$ critical point and, after analytic changes of coordinates (and rescaling the local base coordinate by a nonzero unit), takes the standard form $XY=s$.  Thus Picard--Lefschetz theory gives a single right-handed Dehn twist.  The inverse image of the collision arc is nonseparating, so this twist is precisely $t_{a_\tau}$.
\end{proof}

\section{The three $3+3$ degenerations}
At $t=0$,
\[
q_0(u)=u(u-1)^2.
\]
The three line branch points are $0,1,0$, while the quartic points limit to $0,1,1$.  Thus the six branch points split as
\[
\{0,0,0\}\sqcup\{1,1,1\}.
\]
More precisely, one root near zero is $-t+O(t^2)$, whereas the two roots near one are
\[
1+(3\pm2\sqrt2)t+O(t^2).
\]
The following observation converts these first-order expansions into the braid statement and will also be used at $t=1$ and $t=\infty$.

\begin{figure}[ht]
\centering
\begin{tikzpicture}[scale=1.0,every node/.style={font=\small}]
\begin{scope}[xshift=-3.6cm]
\draw[rounded corners,fill=lineblue!5] (-2,-1.25) rectangle (2,1.25);
\foreach \p in {(-1.25,.25),(-.95,-.35),(-.65,.4)} \fill[lineblue] \p circle (3pt);
\foreach \p in {(.65,-.25),(.95,.38),(1.3,-.38)} \fill[quartic] \p circle (3pt);
\draw[linegreen,very thick] (0,0) ellipse (1.55 and .9);
\node at (0,-1.55) {regular nearby fiber};
\node[linegreen] at (0,1.55) {$\delta$ separates $3+3$};
\end{scope}
\draw[->,very thick] (-.8,0)--(.8,0) node[midway,above] {$t\to0$};
\begin{scope}[xshift=3.6cm]
\draw[rounded corners,fill=lineblue!5] (-2,-1.25) rectangle (2,1.25);
\foreach \p in {(-1.05,0),(-.9,.12),(-.9,-.12)} \fill[lineblue] \p circle (3pt);
\foreach \p in {(.9,0),(1.05,.12),(1.05,-.12)} \fill[quartic] \p circle (3pt);
\draw[linegreen,very thick] (0,0) ellipse (1.55 and .9);
\node at (0,-1.55) {$3+3$ collision};
\end{scope}
\end{tikzpicture}
\caption{Schematic $3+3$ degeneration.  A positive loop in the base rotates both triples.  On the six-punctured sphere the resulting class is $T_\delta^2$; Xiao's local holomorphic Lefschetz model for the relatively minimal fibration selects the separating lift $t_d$ as the upstairs monodromy.}
\label{fig:33}
\end{figure}
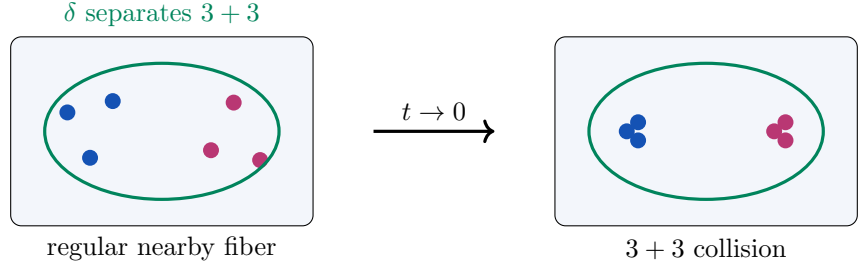

\begin{lemma}[Three-point scaling lemma]\label{lem:scaling}
Suppose three points in a complex disk have expansions
\[
 z_j(s)=z_0+s a_j+O(s^2),\qquad j=1,2,3,
\]
with $a_1,a_2,a_3$ distinct.  For sufficiently small $|s|$, a positive circuit of $s$ about $0$ gives the positive full twist on the three points, equivalently the positive Dehn twist about the boundary of a disk containing them.  Indeed, the homotopy
\[
 z_j^{(r)}(s)=z_0+s a_j+r\,O(s^2),\qquad 0\le r\le1,
\]
remains in the configuration space of three distinct points for $|s|$ small.  Hence the braid is isotopic to $z_0+s a_j$.  As $s=\varepsilon e^{i\theta}$, $0\le\theta\le2\pi$, this is one positive rigid rotation, namely the central full twist $\Delta_3^2$.
\end{lemma}

\begin{proposition}\label{prop:sep}
Let $\delta_0,\delta_1,\delta_\infty$ be circles on the six-punctured quotient sphere separating the two three-point clusters at $0,1,\infty$.  With the complex orientation on the base, the corresponding local spherical monodromies are
\[
T_{\delta_0}^2,\qquad T_{\delta_1}^2,\qquad T_{\delta_\infty}^2.
\]
For Xiao's relatively minimal fibration, the local holomorphic Lefschetz models determine the corresponding upstairs monodromies to be the positive separating twists
\[
t_{d_0},\qquad t_{d_1},\qquad t_{d_\infty},
\]
where $d_j$ is the connected inverse image of $\delta_j$.  Equivalently, each $t_{d_j}$ is the preimage selected by the actual holomorphic degeneration among the two mapping classes lying over $T_{\delta_j}^2$.
\end{proposition}
\begin{proof}
We begin with the spherical braid and its sign.

At $t=0$, the three line branch points are $0,1,t$.  The roots of $q_t$ have expansions
\[
 -t+O(t^2),\qquad
 1+(3+2\sqrt2)t+O(t^2),\qquad
 1+(3-2\sqrt2)t+O(t^2).
\]
Thus the cluster at $0$ has relative coordinates
\[
 0,\quad t,\quad -t+O(t^2),
\]
while the cluster at $1$ has relative coordinates
\[
 0,\quad (3+2\sqrt2)t+O(t^2),\quad (3-2\sqrt2)t+O(t^2).
\]
Lemma~\ref{lem:scaling} shows that a positively oriented loop in $t$ gives one positive full twist in each cluster.

At $t=1$, put $s=t-1$.  The two roots near $0$ and the root near $1$ are
\[
 (3+2\sqrt2)s+O(s^2),\qquad
 (3-2\sqrt2)s+O(s^2),\qquad
 1-s+O(s^2),
\]
respectively.  Together with the fixed branch points $0,1,1+s$, the two clusters again have the form $z_0+s\{a_1,a_2,a_3\}+O(s^2)$ with distinct coefficients, so Lemma~\ref{lem:scaling} gives one positive full twist for each cluster.

At $t=\infty$, use the holomorphic local coordinate $s=1/t$ and the fiber coordinate $v=y/z$.  Restricting the quartic gives
\[
Q_s(v)=C(s,v,1)
=s^3v+s^3-2s^2v^2+2s^2v-2s^2+s v^3-sv^2-sv+s-2v^3+4v^2-2v.
\]
The three line branch points are $v=0,s,1$.  The roots of $Q_s$ satisfy
\[
 \frac12s+O(s^2),\qquad
 1+is+O(s^2),\qquad
 1-is+O(s^2).
\]
Hence the clusters at $0$ and $1$ again satisfy Lemma~\ref{lem:scaling}.  Since $s=1/t$ is a holomorphic coordinate at infinity, a positive small loop about $\infty$ is a positive loop in $s$, so the twists are positive here as well.

For any of the three special values, choose disjoint disks containing the two clusters.  Their boundary circles are isotopic in the six-punctured sphere: they cobound an annulus containing no punctures.  Each cluster full twist is therefore the same positive Dehn twist $T_\delta$ about this isotopy class, and the total spherical monodromy is
\[
 T_\delta T_\delta=T_\delta^2.
\]

It remains to identify the lift.  Because $\delta$ encloses three branch points, the monodromy of the double cover along $\delta$ is nontrivial, so $d=p^{-1}(\delta)$ is connected.  The double cover of either complementary three-punctured disk has Euler characteristic $2-3=-1$ and one boundary component, hence is a one-holed torus.  Thus $d$ is separating.  On an annular neighborhood of $d$ choose coordinates in which the covering is
\[
 p:(\theta,r)\longmapsto(2\theta,r).
\]
A positive Dehn twist upstairs is $(\theta,r)\mapsto(\theta+2\pi\phi(r),r)$; downstairs it becomes
\[
 (2\theta,r)\longmapsto(2\theta+4\pi\phi(r),r),
\]
which is precisely $T_\delta^2$.  Thus $q(t_d)=T_\delta^2$.

To determine which of the two lifts is realized, one must use the local Lefschetz model, not merely the Birman--Hilden quotient.  The calculation above determines the spherical class $T_\delta^2$ directly from the branch motion; the choice between its two genus-two lifts uses the local holomorphic model of Xiao's relatively minimal fibration.  In Xiao's construction, and in the branch-cover description recalled in \cite[Example~8]{AM}, after the vertical $(-1)$-curves are contracted each of the values $0,1,\infty$ contributes a single ordinary node; see \cite{Xiao} and \cite[Example~8]{AM}.  In holomorphic coordinates centered at that node and in a positively oriented local coordinate $s$ on the base, the map has the standard form
\[
 xy=s
\]
(up to multiplication of $s$ by a nonzero holomorphic unit).  Picard--Lefschetz theory therefore gives a single right-handed Dehn twist about the vanishing cycle.  The branch-cluster analysis above identifies that vanishing cycle with the connected lift $d=p^{-1}(\delta)$: the two one-holed tori lie on opposite sides of the neck that is pinched as $s\to0$.  Hence the actual local monodromy is $t_d$.  The alternative lift $\iota t_d$ has the same spherical image but is not the local monodromy of this single nodal smoothing.  The following proposition makes this lift discrepancy explicit in standard Artin-chain coordinates.
\end{proof}

\subsection*{The lift ambiguity for a $3+3$ braid}
The geometric Lefschetz lift differs from the lift obtained by applying the Artin-chain homomorphism generator by generator.  This distinction already appears in the local $3+3$ calculation.  On a standard six-punctured sphere choose consecutive arcs $\eta_1,\ldots,\eta_5$ and let $c_i$ be their lifts to the genus-two double cover.  Put $T_i=t_{c_i}$ and let
\[
 \Phi:B_6\longrightarrow\Mod(\Sigma_2),\qquad \Phi(\sigma_i)=T_i,
\]
be the standard Artin-chain homomorphism.  Let $d$ be the connected lift of a circle $\delta$ enclosing the first three branch points.

\begin{proposition}[The $3+3$ lift discrepancy]\label{prop:deckcorrection-main}
For
\[
 \Omega=(\sigma_1\sigma_2)^3(\sigma_4\sigma_5)^3,
\]
one has
\begin{equation}\label{eq:main-rawlift}
 \Phi(\Omega)=(T_1T_2)^3(T_4T_5)^3=\iota\,t_d,
\end{equation}
where $\iota$ is the hyperelliptic involution, whereas the actual local Lefschetz monodromy is
\begin{equation}\label{eq:main-geometricseplift}
 t_d=(T_1T_2)^6=(T_4T_5)^6.
\end{equation}
Thus the natural generator-by-generator lift of the disk braid is not the geometric Lefschetz lift.
\end{proposition}
\begin{proof}
The two factors $(\sigma_1\sigma_2)^3$ and $(\sigma_4\sigma_5)^3$ are the positive full twists on the two three-point clusters.  On the six-punctured sphere both become the same positive twist $T_\delta$ about the common three-puncture boundary, so $\Omega$ projects to $T_\delta^2$, in agreement with Proposition~\ref{prop:sep}.  The two-chain relation on either one-holed torus bounded by $d$ gives
\[
 (T_1T_2)^6=t_d=(T_4T_5)^6.
\]
Choose a symplectic basis adapted to the two sides of $d$.  On the first handle,
\[
 [T_1]=\begin{pmatrix}1&1\\0&1\end{pmatrix},\qquad
 [T_2]=\begin{pmatrix}1&0\\-1&1\end{pmatrix},
\]
so $[(T_1T_2)^3]=-I_2$; the analogous computation on the second handle gives $[(T_4T_5)^3]=-I_2$.  Hence $[\Phi(\Omega)]=-I_4$, whereas the separating twist $t_d$ acts trivially on $H_1$.  Both classes project to $T_\delta^2$, and the kernel of the Birman--Hilden quotient is $\{1,\iota\}$ with $[\iota]=-I_4$.  Therefore $\Phi(\Omega)=\iota t_d$.  Proposition~\ref{prop:sep} identifies $t_d$, not $\iota t_d$, as the monodromy of the local holomorphic node.
\end{proof}

\section{The ordered geometric global monodromy and an explicit geometric basis}\label{sec:ordered}
Let
\[
B=\mathbb P^1\setminus\{0,1,\infty,\tau_L,\tau_R,\tau_U,\tau_D\}.
\]
Fix the regular value
\[
t_*=\frac14.
\]
At $t_*$ the three moving roots of $q_{t_*}$ are approximately
\[
-0.0885907339,\qquad1.0420816704,\qquad3.0465090635.
\]
Together with $0,t_*,1$, this gives six distinct real branch points.  We label them
\[
b_1<b_2<b_3<b_4<b_5<b_6.
\]
Thus
\[
(b_1,b_2,b_3,b_4,b_5,b_6)
=(-0.0885907,0,1/4,1,1.0420817,3.0465091).
\]
Choose a geometric basis of positively oriented based loops
$\lambda_1,\ldots,\lambda_7$ around the seven critical values, with
\[
\lambda_1\cdots\lambda_7=1
\quad\text{in }\pi_1(B,t_*).
\]
Parallel transport along the chosen vanishing paths places all seven vanishing cycles on one fixed labeled genus-two fiber.

\begin{proposition}[Global product relation]\label{prop:globalproduct}
For the above geometric basis the genus-two monodromies satisfy
\[
\mu_1\cdots\mu_7=1\qquad\text{in }\Mod(\Sigma_2).
\]
The seven factors consist of the four nonseparating twists
$t_{a_L},t_{a_R},t_{a_U},t_{a_D}$ and the three separating twists
$t_{d_0},t_{d_1},t_{d_\infty}$.
\end{proposition}
\begin{proof}
Over $B$ the relatively minimal fibration is a smooth surface bundle, so parallel transport gives
\[
\rho:\pi_1(B,t_*)\longrightarrow\Mod(\Sigma_2).
\]
Applying $\rho$ to the relation $\lambda_1\cdots\lambda_7=1$ gives the product relation.  Propositions~\ref{prop:half} and~\ref{prop:sep} identify the seven local factors and their signs.
\end{proof}

Thus every geometric basis gives a positive factorization of type $(4,3)$.  For a fixed identification of the reference fiber, two distinguished systems are related by Hurwitz moves; changing the identification also allows simultaneous conjugation.

For the explicit calculation, fix the star system below.  Appendix~\ref{app:independent-check} records another collection of paths for comparison.

Put
\[
 z=\frac{t}{t+3i},\qquad z_*=\frac{t_*}{t_*+3i},\qquad t_*=\frac14.
\]
This is a M\"obius coordinate on the base; in particular $t=\infty$ is the finite point $z=1$.  For each critical value $\xi\in\{0,1,\infty,\tau_L,\tau_R,\tau_U,\tau_D\}$, let $r_\xi$ be the straight segment from $z_*$ to $z(\xi)$.  Choose
\[
 \rho_\xi=\frac1{10}\min\Bigl(\{|z(\xi)-z(\zeta)|:\zeta\ne\xi\}\cup
 \{\operatorname{dist}(z(\xi),r_\zeta):\zeta\ne\xi\}\Bigr),
\]
where in the second set the distance is taken only to those nonincident straight segments for which the displayed distance is positive.  Starting at $z_*$, follow $r_\xi$ to the first point of the circle $|z-z(\xi)|=\rho_\xi$, traverse that circle once counterclockwise, and return along the reverse straight segment.  Denote the resulting based meridian by $\lambda_\xi$.  Any smaller positive choice of the radii gives the same based homotopy classes.  Near $z_*$ the seven outgoing arcs may be separated by an arbitrarily small isotopy, preserving their cyclic order, to obtain a distinguished geometric basis.  The braid elements attached to the individual based meridians are unchanged by this separation.

To six decimal places, the arguments of the displacement vectors $z(\xi)-z_*$, measured counterclockwise from the positive real axis, are
\[
\begin{array}{c|rrrrrrr}
\xi&\tau_U&\infty&\tau_L&0&\tau_D&1&\tau_R\\ \hline
\arg(z(\xi)-z_*)&3.776^\circ&4.764^\circ&93.046^\circ&94.764^\circ&223.306^\circ&293.199^\circ&294.731^\circ.
\end{array}
\]
All seven arguments are distinct.  Hence the open straight segments from $z_*$ are pairwise disjoint; in particular no such segment passes through another critical value.  Since there are only finitely many segments and endpoints, the terminal circles can be chosen mutually disjoint and disjoint from every nonincident segment.  A standard arbitrarily small separation of the seven arcs in a neighborhood of $z_*$ therefore produces a distinguished star system without changing their cyclic order.

Starting immediately after the ray to $\tau_D$ and proceeding counterclockwise, that cyclic order is
\begin{equation}\label{eq:orderedbasis}
 \lambda_D,\ \lambda_1,\ \lambda_R,\ \lambda_U,\ \lambda_\infty,\ \lambda_L,\ \lambda_0.
\end{equation}
With the standard boundary convention we therefore have
\[
 \lambda_D\lambda_1\lambda_R\lambda_U\lambda_\infty\lambda_L\lambda_0=1
 \quad\text{in }\pi_1\!\left(\mathbb P^1\setminus\{0,1,\infty,\tau_L,\tau_R,\tau_U,\tau_D\},t_*\right).
\]
Here the displayed word records the positively oriented meridians in the counterclockwise order of the outgoing paths at the basepoint, with the boundary orientation convention used above.  This is the usual product relation for a punctured sphere.

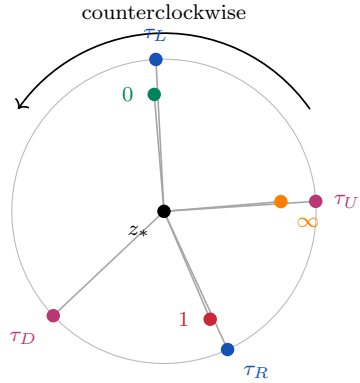
\begin{figure}[ht]
\centering
\begin{tikzpicture}[scale=1.15,every node/.style={font=\scriptsize}]
  \coordinate (O) at (0,0);
  \draw[gray!55] (O) circle (1.75);
  \draw[->,line width=.65pt] (35:2.05) arc[start angle=35,end angle=145,radius=2.05];
  \node at (90:2.30) {counterclockwise};

  \draw[gray!70,line width=.6pt] (O)--(3.776:1.75);
  \draw[gray!70,line width=.6pt] (O)--(4.764:1.35);
  \draw[gray!70,line width=.6pt] (O)--(93.046:1.75);
  \draw[gray!70,line width=.6pt] (O)--(94.764:1.35);
  \draw[gray!70,line width=.6pt] (O)--(223.306:1.75);
  \draw[gray!70,line width=.6pt] (O)--(293.199:1.35);
  \draw[gray!70,line width=.6pt] (O)--(294.731:1.75);

  \fill[black] (O) circle (2.2pt) node[below left=2pt] {$z_*$};
  \fill[quartic] (3.776:1.75) circle (2.2pt) node[right=3pt] {$\tau_U$};
  \fill[lineorange] (4.764:1.35) circle (2.2pt) node[below right=3pt] {$\infty$};
  \fill[lineblue] (93.046:1.75) circle (2.2pt) node[above=3pt] {$\tau_L$};
  \fill[linegreen] (94.764:1.35) circle (2.2pt) node[left=4pt] {$0$};
  \fill[quartic] (223.306:1.75) circle (2.2pt) node[below left=3pt] {$\tau_D$};
  \fill[linered] (293.199:1.35) circle (2.2pt) node[left=4pt] {$1$};
  \fill[lineblue] (294.731:1.75) circle (2.2pt) node[below right=3pt] {$\tau_R$};
\end{tikzpicture}
\caption{Angular diagram of the outgoing paths of the fixed star system at $z_*$.  The seven directions are the arguments listed above; the radii are staggered only to separate the labels.  Starting immediately after the ray to $\tau_D$ and moving counterclockwise gives $\lambda_D,\lambda_1,\lambda_R,\lambda_U,\lambda_\infty,\lambda_L,\lambda_0$, which is the order used in \eqref{eq:orderedbasis}.}
\label{fig:star-system}
\end{figure}

Let
\[
 \mathcal W_j:=\rho(\lambda_j)\in\Mod(\Sigma_2),
\]
where $\rho$ is the monodromy representation of the actual double-cover family and the indices are taken in the order \eqref{eq:orderedbasis}.  These are the geometric monodromies obtained by parallel transport.  The braid calculation below puts them in explicit Artin coordinates.

\begin{theorem}[Ordered geometric monodromy relation for the fixed star system]\label{thm:orderedexplicit}
For the distinguished basis \eqref{eq:orderedbasis}, the seven actual local monodromies satisfy
\begin{equation}\label{eq:orderedidentity}
 \boxed{\mathcal W_D\mathcal W_1\mathcal W_R\mathcal W_U\mathcal W_\infty\mathcal W_L\mathcal W_0=1\quad\text{in }\Mod(\Sigma_2).}
\end{equation}
They consist of four positive nonseparating twists and three positive separating twists.
\end{theorem}

\begin{proof}
The star meridians form a geometric basis, so their product is the identity in the fundamental group of the punctured base.  Parallel transport gives a homomorphism to $\Mod(\Sigma_2)$ and hence \eqref{eq:orderedidentity}.  Proposition~\ref{prop:half} identifies the four tangency factors, and Proposition~\ref{prop:sep} identifies the three separating factors.
\end{proof}

\subsection{Vanishing cycles on the reference fiber}\label{subsec:star-vanishing-cycles}
We next place the seven vanishing cycles on the reference fiber.  On the branch sphere over $t_*=1/4$, let
\[
 b_1<b_2<b_3<b_4<b_5<b_6
\]
be the ordered branch points fixed above, and let $\eta_i$, $1\le i\le5$, be the standard embedded arc joining $b_i$ to $b_{i+1}$ and disjoint from the remaining branch points.  Write
\[
 \varpi:\Sigma_2\longrightarrow S^2
\]
for the hyperelliptic double cover of the reference fiber.

For $j\in\{D,R,U,L\}$, transport the local collision arc at $\tau_j$ back to the reference branch sphere along the corresponding path of the distinguished star system, and denote the resulting arc by $\alpha_j$.  Similarly, at $0,1,\infty$, transport the boundary of either three-point cluster back to the reference sphere and denote the resulting three-puncture circles by
\[
 \Delta_0,\qquad \Delta_1,\qquad \Delta_\infty.
\]
Their hyperelliptic lifts
\[
 a_j=\varpi^{-1}(\alpha_j),\qquad d_k=\varpi^{-1}(\Delta_k)
\]
are connected simple closed curves on $\Sigma_2$, nonseparating in the first case and separating in the second.

\begin{proposition}\label{prop:star-vanishing-curves}
For the distinguished basis \eqref{eq:orderedbasis},
\[
\begin{aligned}
 \mathcal W_D&=t_{a_D},&
 \mathcal W_1&=t_{d_1},&
 \mathcal W_R&=t_{a_R},&
 \mathcal W_U&=t_{a_U},\\
 \mathcal W_\infty&=t_{d_\infty},&
 \mathcal W_L&=t_{a_L},&
 \mathcal W_0&=t_{d_0}.
\end{aligned}
\]
Consequently \eqref{eq:orderedidentity} can be written
\[
 \boxed{t_{a_D}t_{d_1}t_{a_R}t_{a_U}t_{d_\infty}t_{a_L}t_{d_0}=1.}
\]
\end{proposition}

\begin{proof}
At each of the four roots of $R$, Proposition~\ref{prop:half} identifies the local monodromy with the positive half twist about the collision arc.  Parallel transport along the chosen star path carries that arc to $\alpha_j$, and its lift is the nonseparating curve $a_j$.  At $0,1,\infty$, Proposition~\ref{prop:sep} identifies the local Lefschetz monodromy with the positive twist about the connected lift of the corresponding three-puncture circle.  Transport to the reference fiber gives $d_0,d_1,d_\infty$.  The order is the order \eqref{eq:orderedbasis}.
\end{proof}

We now compute braid representatives for the same star system, thereby putting the transported curves in standard Artin coordinates.

Along each based meridian of the fixed star system, we move straight from $z_*$ to the terminal circle, go once counterclockwise around the circle, and return along the same segment.  The small separation near $z_*$ used to display the geometric basis is not part of this continuation.  Different generic projections or parametrizations may change the written Artin word, but not the braid it represents.  For the real projection we use the disk coordinate
\begin{equation}\label{eq:wcoord}
 w=\frac{1}{u-20i},
\end{equation}
which sends $u=\infty$ to $w=0$.  We make the crossing convention explicit.  Along a parametrized path let $w_p(s),w_q(s)$ be two adjacent trajectories in this disk chart, and put
\[
 d(s)=\operatorname{Re}(w_p(s)-w_q(s)),\qquad
 h(s)=\operatorname{Im}(w_p(s)-w_q(s)).
\]
At a transverse projected crossing $s=s_0$, so that $d(s_0)=0$ and $d'(s_0)\ne0$, we record the exponent
\begin{equation}\label{eq:crossing-sign}
 \varepsilon=-\operatorname{sgn}\bigl(d'(s_0)h(s_0)\bigr)\in\{+1,-1\}.
\end{equation}
This convention makes a positive geometric half twist the standard Artin generator $\sigma_i$.  Following the roots with this convention gives the words below.  We write $j$ for $\sigma_j$ and $-j$ for $\sigma_j^{-1}$:
\begingroup\scriptsize
\begin{align*}
 w_D={}&(-5,-4,-3,1,-5,-2,-3,4,3,2,5,-1,3,4,5),\\
 w_1={}&(5,4,3,2,1,2,1,2,4,5,4,2,1,2,4,5,4,-1,-2,-3,-4,-5),\\
 w_R={}&(5,4,3,2,1,-4,-5,-2,-1,-2,-4,3,4,2,1,2,5,4,-1,-2,-3,-4,-5),\\
 w_U={}&(5,4,-1,5,-4,-3,2,3,4,-5,1,-4,-5),\\
 w_\infty={}&(5,4,5,-1,-4,-3,-2,-5,-4,-3,-5,-4,-5,-3,-4,-5,-2,-3,-4,3,2,3,3,2,-1,3,-2,-3,\notag\\
&\qquad 2,1,2,3,2,1,3,2,3,4,3,2,5,4,3,5,4,5,3,4,5,2,3,4,1,-5,-4,-5),\\
 w_L={}&(2,1,2,5,4,5,3,-5,-4,-5,-2,-1,-2),\\
 w_0={}&(2,1,2,5,4,5,2,1,2,5,4,5).
\end{align*}
\endgroup

With these choices, the four tangency words are
\[
 w_D=B_D\sigma_4B_D^{-1},\qquad
 w_R=B_R\sigma_3B_R^{-1},\qquad
 w_U=B_U\sigma_2B_U^{-1},\qquad
 w_L=B_L\sigma_3B_L^{-1},
\]
where
\begingroup\small
\[
\begin{aligned}
 B_D={}&\sigma_5^{-1}\sigma_4^{-1}\sigma_3^{-1}\sigma_1\sigma_5^{-1}\sigma_2^{-1}\sigma_3^{-1},\\
 B_R={}&\sigma_5\sigma_4\sigma_3\sigma_2\sigma_1\sigma_4^{-1}\sigma_5^{-1}
          \sigma_2^{-1}\sigma_1^{-1}\sigma_2^{-1}\sigma_4^{-1},\\
 B_U={}&\sigma_5\sigma_4\sigma_1^{-1}\sigma_5\sigma_4^{-1}\sigma_3^{-1},\\
 B_L={}&\sigma_2\sigma_1\sigma_2\sigma_5\sigma_4\sigma_5.
\end{aligned}
\]
\endgroup
The four arcs of Proposition~\ref{prop:star-vanishing-curves} are therefore represented by
\begin{equation}\label{eq:star-arcs}
 \alpha_D=B_D(\eta_4),\qquad
 \alpha_R=B_R(\eta_3),\qquad
 \alpha_U=B_U(\eta_2),\qquad
 \alpha_L=B_L(\eta_3).
\end{equation}
Their endpoint pairs are respectively
\[
 \{b_1,b_5\},\qquad \{b_5,b_6\},\qquad
 \{b_1,b_5\},\qquad \{b_1,b_6\}.
\]
In particular $\alpha_D$ and $\alpha_U$ have the same endpoints but represent different transported arcs in the punctured sphere; equation~\eqref{eq:star-arcs}, rather than the endpoint pair alone, records their isotopy classes.

The three separating factors admit an equally concrete description.  Let $\Delta_{ijk}$ denote the boundary of a disk containing precisely $b_i,b_j,b_k$.  The word $w_0$ represents the $3+3$ braid with partition
\[
 \Delta_0=\Delta_{123}.
\]
If
\[
 K_1=\sigma_5\sigma_4\sigma_3\sigma_2\sigma_1,
\]
then the displayed word $w_1$ is $K_1\Omega K_1^{-1}$, where
\[
 \Omega=(\sigma_1\sigma_2)^3(\sigma_4\sigma_5)^3.
\]
Hence
\[
 \Delta_1=K_1(\Delta_{123}),
\]
and $\Delta_1$ separates $\{b_1,b_2,b_6\}$ from $\{b_3,b_4,b_5\}$.  For the factor at infinity, Appendix~\ref{app:star-infinity-check} gives an explicit decomposition
\[
 w_\infty=K_\infty m_\infty K_\infty^{-1}
\]
and verifies, by the faithful Artin action, that $m_\infty=(\sigma_2\sigma_3)^6$ in $B_6$.  The induced permutation of $K_\infty$ sends $\{2,3,4\}$ to $\{1,3,6\}$.  Hence the transported three-puncture circle $\Delta_\infty$ separates
\[
 \{b_1,b_3,b_6\}\quad\text{from}\quad\{b_2,b_4,b_5\}.
\]
Here $(\sigma_2\sigma_3)^3=\Delta_3^2$ is the positive full twist on the three strands in positions $2,3,4$, hence represents the positive Dehn twist about the boundary of a disk containing those three punctures.  Therefore $(\sigma_2\sigma_3)^6=(\Delta_3^2)^2$ maps, under the disk-to-sphere quotient, to $T_\delta^2$.  This is the same spherical mapping class that appears in Proposition~\ref{prop:sep}; the corresponding disk-braid word need not agree with $\Omega$ in Proposition~\ref{prop:deckcorrection-main}.

Figure~\ref{fig:separating-downstairs} shows the three transported
three-puncture circles on the reference branch sphere.  The drawings are
schematic: only the labeled partition of the six branch points is relevant.

\begin{figure}[ht]
\centering
\begin{tikzpicture}[scale=.92,every node/.style={font=\scriptsize}]
\begin{scope}[xshift=0cm]
  \draw[line width=.65pt] (0,0) circle (1.72);
  \foreach \x/\lab in {-1/$b_1$,0/$b_2$,1/$b_3$}
    {\fill (\x,.62) circle (2pt) node[above=3pt] {\lab};}
  \foreach \x/\lab in {-1/$b_4$,0/$b_5$,1/$b_6$}
    {\fill (\x,-.68) circle (2pt) node[below=3pt] {\lab};}
  \draw[linered,line width=1.15pt]
    (0,.68) ellipse (1.38 and .58);
  \node[linered] at (0,1.45) {$\Delta_0$};
  \node at (0,-2.02) {$\{b_1,b_2,b_3\}\mid\{b_4,b_5,b_6\}$};
\end{scope}

\begin{scope}[xshift=5.0cm]
  \draw[line width=.65pt] (0,0) circle (1.72);
  \foreach \x/\lab in {-1/$b_1$,0/$b_2$,1/$b_6$}
    {\fill (\x,.62) circle (2pt) node[above=3pt] {\lab};}
  \foreach \x/\lab in {-1/$b_3$,0/$b_4$,1/$b_5$}
    {\fill (\x,-.68) circle (2pt) node[below=3pt] {\lab};}
  \draw[lineblue,line width=1.15pt]
    (0,.68) ellipse (1.38 and .58);
  \node[lineblue] at (0,1.45) {$\Delta_1$};
  \node at (0,-2.02) {$\{b_1,b_2,b_6\}\mid\{b_3,b_4,b_5\}$};
\end{scope}

\begin{scope}[xshift=10.0cm]
  \draw[line width=.65pt] (0,0) circle (1.72);
  \foreach \x/\lab in {-1/$b_1$,0/$b_3$,1/$b_6$}
    {\fill (\x,.62) circle (2pt) node[above=3pt] {\lab};}
  \foreach \x/\lab in {-1/$b_2$,0/$b_4$,1/$b_5$}
    {\fill (\x,-.68) circle (2pt) node[below=3pt] {\lab};}
  \draw[linegreen,line width=1.15pt]
    (0,.68) ellipse (1.38 and .58);
  \node[linegreen] at (0,1.45) {$\Delta_\infty$};
  \node at (0,-2.02) {$\{b_1,b_3,b_6\}\mid\{b_2,b_4,b_5\}$};
\end{scope}
\end{tikzpicture}
\caption{Schematic representatives of the three separating curves downstairs.
In the three panels the punctures have been repositioned by an isotopy of the
sphere so that the enclosed triple is visible; the labels, not their planar
positions, determine the partition.  The spherical monodromies are
$T_{\Delta_0}^2$, $T_{\Delta_1}^2$, and $T_{\Delta_\infty}^2$.
Since each side of $\Delta_j$ contains three branch points, its preimage under
the hyperelliptic double cover is a connected separating curve $d_j$; the two
complementary components are once-punctured tori.}
\label{fig:separating-downstairs}
\end{figure}
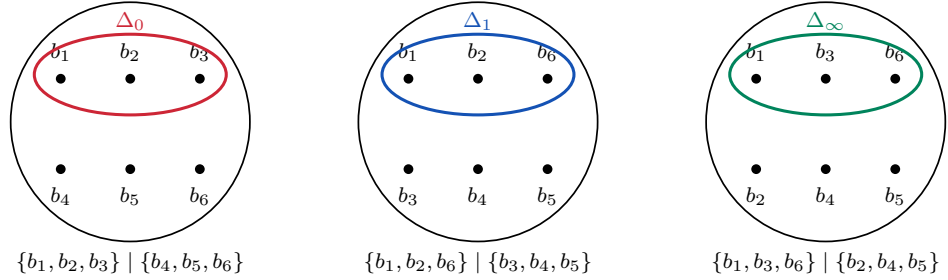
Thus the seven curves on the reference branch sphere are summarized by
\[
\begin{array}{c|c|c}
\text{critical value}&\text{curve downstairs}&\text{branch data}\\ \hline
\tau_D&\alpha_D=B_D(\eta_4)&b_1\leftrightarrow b_5\\
1&\Delta_1&\{b_1,b_2,b_6\}\mid\{b_3,b_4,b_5\}\\
\tau_R&\alpha_R=B_R(\eta_3)&b_5\leftrightarrow b_6\\
\tau_U&\alpha_U=B_U(\eta_2)&b_1\leftrightarrow b_5\\
\infty&\Delta_\infty&\{b_1,b_3,b_6\}\mid\{b_2,b_4,b_5\}\\
\tau_L&\alpha_L=B_L(\eta_3)&b_1\leftrightarrow b_6\\
0&\Delta_0&\{b_1,b_2,b_3\}\mid\{b_4,b_5,b_6\}.
\end{array}
\]
Taking hyperelliptic inverse images gives
\[
 a_D,\ d_1,\ a_R,\ a_U,\ d_\infty,\ a_L,\ d_0
\]
on the fixed genus-two fiber, in the order of Theorem~\ref{thm:orderedexplicit}.

The exponent sums are $1,12,1,1,12,1,12$, as expected from the four tangencies and the three $3+3$ degenerations.  Using $300,700,1500,$ and $3000$ subdivisions gives the same braid elements under the faithful Artin action $B_6\to\operatorname{Aut}(F_6)$.  The continuation is described in Appendix~\ref{app:braidcode}.

We use
\[
\Phi:B_6\to\Mod(\Sigma_2),\qquad
\pi:B_6\to\Mod(S^2,\mathbf b),\qquad
q:\Mod(\Sigma_2)\to\Mod(S^2,\mathbf b),
\]
with $q\circ\Phi=\pi$.  We write $\widehat w_j=\Phi(w_j)$ for the Artin-chain lift of a printed disk-braid word.  This need not equal the geometric monodromy $\mathcal W_j$ at a $3+3$ value; the two possible lifts differ by the central involution $\iota$.

Let $\widehat w_j=\Phi(w_j)$.  At the three $3+3$ values the deck-involution correction can be determined explicitly.  Put
\[
 \Theta:=T_1T_2T_3T_4T_5^2T_4T_3T_2T_1=\iota,
\]
the standard five-chain expression for the genus-two hyperelliptic involution; see, for example, \cite{BH,FM}.  Exact multiplication in the symplectic representation gives
\begin{equation}\label{eq:sep-corrections}
 [\widehat w_1]=-I_4,\qquad [\widehat w_\infty]=I_4,\qquad [\widehat w_0]=-I_4.
\end{equation}
Each $\widehat w_j$ projects to the corresponding spherical class $T_{\delta_j}^2$, whereas the actual separating Lefschetz twist acts trivially on $H_1$.  Since the kernel of the Birman--Hilden quotient is $\{1,\iota\}$ and $[\iota]=-I_4$, equation \eqref{eq:sep-corrections} fixes the correction uniquely:
\begin{equation}\label{eq:literal-sep-lifts}
 \mathcal W_1=\Theta\widehat w_1,\qquad
 \mathcal W_\infty=\widehat w_\infty,\qquad
 \mathcal W_0=\Theta\widehat w_0,
\end{equation}
for the displayed words.  At the four tangencies the corresponding encoding is
\begin{equation}\label{eq:literal-nonsep-lifts}
 \mathcal W_D=\widehat w_D,\qquad \mathcal W_R=\widehat w_R,\qquad
 \mathcal W_U=\widehat w_U,\qquad \mathcal W_L=\widehat w_L.
\end{equation}
Equations \eqref{eq:literal-sep-lifts}--\eqref{eq:literal-nonsep-lifts} now give words in $T_1,\ldots,T_5$ for the seven geometric factors.

It remains to check the product of the braid words, first in the spherical quotient and then after lifting to genus two.  Put
\begin{equation}\label{eq:BH-commuting}
 q\circ\Phi=\pi.
\end{equation}
The product is not the identity in the disk braid group.  The exponent sums of the seven extracted words, in the global order, are
\[
 1,\ 12,\ 1,\ 1,\ 12,\ 1,\ 12,
\]
so for
\begin{equation}\label{eq:raw-braid-product}
 b:=w_Dw_1w_Rw_Uw_\infty w_Lw_0\in B_6
\end{equation}
one has
\begin{equation}\label{eq:raw-exponent-sum}
 \operatorname{exp}(b)=1+12+1+1+12+1+12=40.
\end{equation}
Since the abelianization $B_6\to\mathbb Z$ sends every $\sigma_i$ to $1$, equation \eqref{eq:raw-exponent-sum} shows in particular that $b\ne1$ in the ordinary disk braid group.  The monodromy relation is instead a relation after passage to the six-punctured sphere.

\begin{proposition}[Exact spherical-braid and lifted-product check]\label{prop:spherical-braid-check}
For the seven Artin strings obtained above, the product \eqref{eq:raw-braid-product} satisfies
\begin{equation}\label{eq:spherical-braid-identity}
 \pi(b)=1\qquad\text{in }\Mod(S^2,\mathbf b).
\end{equation}
Moreover its generator-by-generator hyperelliptic lift satisfies
\begin{equation}\label{eq:raw-lift-identity}
 \Phi(b)=1\qquad\text{in }\Mod(\Sigma_2).
\end{equation}
Consequently the corrected geometric Lefschetz factors satisfy
\begin{equation}\label{eq:literal-exact-product}
 \mathcal W_D\mathcal W_1\mathcal W_R\mathcal W_U\mathcal W_\infty\mathcal W_L\mathcal W_0=1
 \qquad\text{in }\Mod(\Sigma_2).
\end{equation}
\end{proposition}

\begin{proof}
For \eqref{eq:spherical-braid-identity}, use the standard faithful Artin action on
\[
 F_6=\langle x_1,\ldots,x_6\rangle,
\]
with
\[
 \sigma_i:x_i\longmapsto x_ix_{i+1}x_i^{-1},\qquad
 x_{i+1}\longmapsto x_i,
\]
and all other free generators fixed.  Passing from the disk to the six-punctured sphere replaces the free group by
\[
 G:=\langle x_1,\ldots,x_6\mid x_1x_2\cdots x_6=1\rangle
 \cong \pi_1(S^2\setminus\mathbf b).
\]
Because an Artin braid preserves the boundary word $x_1\cdots x_6$ up to equality, its automorphism of $F_6$ descends to an automorphism of $G$.  Eliminating $x_6$ identifies
\[
 G\cong F_5=\langle x_1,\ldots,x_5\rangle.
\]
Exact free reduction of the automorphism defined by the displayed braid word $b$ gives
\begin{equation}\label{eq:b-inner-action}
 b_*(x_j)=c\,x_j\,c^{-1}\qquad(1\le j\le5),
\end{equation}
where
\begin{equation}\label{eq:b-conjugator}
 c=x_5^{-1}x_4^{-1}x_3^{-1}x_2^{-1}x_3x_5^{-1}x_4^{-1}x_3^{-1}x_2^{-1}x_1^{-1}.
\end{equation}
Thus the induced outer automorphism of $\pi_1(S^2\setminus\mathbf b)$ is trivial.  Since $b$ already comes from an orientation-preserving punctured-sphere mapping class, its automorphism preserves the peripheral conjugacy classes.  Dehn--Nielsen--Baer therefore gives
\[
 \pi(b)=1,
\]
which proves \eqref{eq:spherical-braid-identity}.

Now set
\[
 \widehat P:=\widehat w_D\widehat w_1\widehat w_R\widehat w_U
 \widehat w_\infty\widehat w_L\widehat w_0
 =\Phi(b).
\]
By \eqref{eq:BH-commuting} and \eqref{eq:spherical-braid-identity},
\[
 q(\widehat P)=\pi(b)=1.
\]
Hence
\begin{equation}\label{eq:raw-lift-kernel}
 \widehat P\in\ker q=\{1,\iota\}.
\end{equation}
The homology action distinguishes the two possibilities.  The four nonseparating words act on $H_1(\Sigma_2;\mathbb Z)$ by
\[
[\widehat w_D]=\begin{pmatrix}-1&4&-2&2\\-1&3&-1&1\\-1&2&0&1\\-1&2&-1&2\end{pmatrix},\quad
[\widehat w_R]=\begin{pmatrix}1&4&0&2\\0&1&0&0\\0&2&1&1\\0&0&0&1\end{pmatrix},
\]
\[
[\widehat w_U]=\begin{pmatrix}3&4&2&2\\-1&-1&-1&-1\\1&2&2&1\\-1&-2&-1&0\end{pmatrix},\quad
[\widehat w_L]=\begin{pmatrix}1&0&0&0\\-1&1&-1&0\\0&0&1&0\\-1&0&-1&1\end{pmatrix}.
\]
Direct multiplication gives
\begin{equation}\label{eq:ordered-homology-check}
 [\widehat w_D][\widehat w_R][\widehat w_U][\widehat w_L]=I_4.
\end{equation}
Together with \eqref{eq:sep-corrections}, and using the actual interleaved order, this yields
\[
 [\widehat P]
 =I_4(-I_4)I_4I_4(I_4)I_4(-I_4)
 =I_4.
\]
On the other hand the hyperelliptic involution satisfies $[\iota]=-I_4$.  Equation \eqref{eq:raw-lift-kernel} therefore forces
\[
 \widehat P=1,
\]
which proves \eqref{eq:raw-lift-identity}.  Only the two elements of $\ker q$ are being distinguished here; no faithfulness statement about the symplectic representation is used.

Finally, the actual Lefschetz factors differ from the raw Artin-chain lifts only at $1$ and $0$.  Equations \eqref{eq:literal-sep-lifts}--\eqref{eq:literal-nonsep-lifts} give
\[
\begin{aligned}
 \mathcal W_D\mathcal W_1\mathcal W_R\mathcal W_U\mathcal W_\infty\mathcal W_L\mathcal W_0
 &=\widehat w_D(\iota\widehat w_1)\widehat w_R\widehat w_U
   \widehat w_\infty\widehat w_L(\iota\widehat w_0)\\
 &=\iota^2\,
   \widehat w_D\widehat w_1\widehat w_R\widehat w_U
   \widehat w_\infty\widehat w_L\widehat w_0.
\end{aligned}
\]
Here we used that $\iota$ is central in $\Mod(\Sigma_2)$.  Since $\iota^2=1$ and $\widehat P=1$, the corrected product is also the identity.  This proves \eqref{eq:literal-exact-product}.
\end{proof}

Proposition~\ref{prop:spherical-braid-check} also explains why the product relation holds after passage to the spherical mapping class group, although the corresponding word is nontrivial in $B_6$ by \eqref{eq:raw-exponent-sum}.

\section{Fundamental group and the regular-fiber complement}\label{sec:pi1}
We next recall the fundamental-group information needed for fiber sums.  We do not recompute $\pi_1$ from the braid words; we use the facts established in Akhmedov--Monden \cite[Section~5]{AM}.

Let $S$ denote the relatively minimal total space and let $\Sigma\subset S$ be a regular genus-two fiber.  Because the fibration has a section, the usual van Kampen description gives a surjection
\begin{equation}\label{eq:fibersurj}
\pi_1(\Sigma)\twoheadrightarrow\pi_1(S),
\end{equation}
whose kernel is the normal closure of the vanishing cycles.  Thus, for any choice of standard generators
\[
\pi_1(\Sigma)=\langle \bar c_1,\bar d_1,\bar c_2,\bar d_2\mid[\bar c_1,\bar d_1][\bar c_2,\bar d_2]=1\rangle,
\]
one has the presentation
\begin{equation}\label{eq:pi1vanishing}
\pi_1(S)=\pi_1(\Sigma)/\!\!\left\langle\!\left\langle
 a_L,a_R,a_U,a_D,d_0,d_1,d_\infty
\right\rangle\!\right\rangle,
\end{equation}
where in \eqref{eq:pi1vanishing} the symbols denote loops represented by the seven geometric vanishing cycles constructed above and assembled in Theorem~\ref{thm:orderedexplicit}.  Formula \eqref{eq:pi1vanishing} is useful even before choosing literal words for those curves.

For Xiao's fibration the quotient is known explicitly:
\begin{proposition}\label{prop:pi1}
For the relatively minimal Xiao fibration considered here,
\[
\pi_1(S)\cong\Z^2.
\]
Moreover the inclusion of the complement of a regular fiber induces an isomorphism
\[
\pi_1(S\setminus\nu\Sigma)\cong\pi_1(S)\cong\Z^2,
\]
and a meridian of $\Sigma$ is null-homotopic in $S\setminus\nu\Sigma$.
\end{proposition}
\begin{proof}
Akhmedov--Monden identify Xiao's total space as $T^2\times S^2\#3\overline{\CP}^{\,2}$ (equivalently, the indicated three-fold blowup of the ruled surface) and hence its fundamental group is $\Z^2$ \cite[Section~2.5]{AM}.  They further observe that the fundamental group is carried by a regular fiber and that the existence of a sphere section kills the meridian of the fiber in the complement; consequently
\[
\pi_1(S\setminus\nu\Sigma)\cong\pi_1(S)\cong\Z^2
\]
\cite[Section~5.1]{AM}.  In the present branch-cover model there are in fact three disjoint $(-1)$-sections, so the null-homotopy of the meridian is visible directly: the punctured section is a disk in $S\setminus\nu\Sigma$ whose boundary is the meridian.
\end{proof}

This is the form needed in fiber-sum calculations: on the Xiao side the boundary meridian is trivial, and the image of the fiber group is $\mathbb Z^2$.  Van Kampen for $(X\setminus\nu F)\cup_\varphi(S\setminus\nu\Sigma)$ therefore adds the relations in the kernel of $\pi_1(\Sigma)\to\mathbb Z^2$ and kills the glued meridian, as in \cite[Section~5]{AM}.

\section{Topological checks from the branch cover}
The same invariants can also be read directly from the branch cover, in parallel with the computation of Akhmedov--Monden for Xiao's model \cite[Example~8]{AM}.  Since $e(Y)=10$ and the seven components of $D$ are smooth rational curves,
\[
e(D)=14,
\]
so
\[
e(\widetilde S)=2e(Y)-e(D)=20-14=6.
\]
Furthermore
\[
D^2=100-16\cdot7=-12,
\]
and the branched double-cover signature formula (see, for example, \cite{BHPV}) gives
\[
\sigma(\widetilde S)=2\sigma(Y)-\frac12D^2=2(-6)+6=-6.
\]
Contracting the three vertical $(-1)$-spheres gives the relatively minimal fibration $S$ with
\[
e(S)=3,\qquad \sigma(S)=-3.
\]
The same Euler characteristic follows from the Lefschetz fibration:
\[
e(S)=e(S^2)e(\Sigma_2)+7=-4+7=3.
\]
Endo's hyperelliptic signature formula \cite{Endo} for $(n,s)=(4,3)$ gives
\[
\sigma(S)=-\frac35n-\frac15s=-3.
\]
Thus
\[
\chi_h(S)=0,\qquad c_1^2(S)=-3.
\]
The three branch-line components not contained in fibers lift to three disjoint $(-1)$-sections, in agreement with the branch-cover description in \cite{AM} and \cite{Xiao}.

\section{Conclusion}
Starting with Xiao's complete quadrangle, we obtain four nonseparating twists from the tangencies and three separating classes from the $3+3$ degenerations, ordered by the chosen star system.  At a $3+3$ value the local holomorphic Lefschetz model chooses the appropriate genus-two lift.  Proposition~\ref{prop:deckcorrection-main} identifies the difference between the natural Artin-chain lift and the Picard--Lefschetz monodromy with the hyperelliptic involution.  In this way the factorization is obtained from the branched-cover geometry itself.  Section~\ref{sec:ordered} gives explicit numerical Artin representatives for the seven factors.

\appendix
\section{Exact algebra checks}\label{app:exact}
This appendix collects the polynomial calculations used in Proposition~\ref{prop:quartic} and in \eqref{eq:disc}.  All computations are over $\mathbb Q$.

The quartic equation can also be recovered directly by linear algebra.  Writing a general degree-four homogeneous polynomial with $15$ coefficients, impose $F(P_i)=0$ for $0\le i\le3$ and $F_x(P_j)=F_y(P_j)=F_z(P_j)=0$ for $4\le j\le6$.  The resulting coefficient matrix has rank $13$ and nullspace basis
\[
z(x-y)(x+y-z)^2,\qquad y(x-z)(x-y+z)^2,
\]
which proves the assertions preceding \eqref{eq:quarticpencil}.

For $f(y,z)=C(1,y,z)$, lexicographic elimination with $y>z$ gives
\[
\operatorname{GB}_{\mathrm{lex}}(f,f_y,f_z)=\{y+z-1,\ z^2-z\}.
\]
Substituting the two solutions $(y,z)=(1,0),(0,1)$ into $f,f_y,f_z$ verifies the two affine singular points.  The missing projective line $x=0$ is handled separately in Proposition~\ref{prop:quartic}.

For the cubic
\[
q_t(u)=t(1-u)(1+u-t)^2+u(1-t)(1-u+t)^2,
\]
the discriminant with respect to $u$ is exactly
\[
\operatorname{Disc}_u(q_t)
=-32t^2(t-1)^2(2t^4-4t^3+12t^2-10t-1).
\]
The quartic factor has
\[
\operatorname{Disc}_t(R)=-1259712,
\]
so it has four distinct roots.  Finally, the six line restrictions listed in Remark~\ref{rem:branchsmooth} are direct substitutions in the homogeneous equation of $C$ and certify all line--quartic intersection multiplicities used to separate the branch divisor.

These are precisely the algebraic identities used above.

\section{Numerical braid-word computation}\label{app:braidcode}
The braid words in Section~\ref{sec:ordered} come from following the branch points along the star meridians fixed there.  We used the following numerical procedure:
\begin{enumerate}
\item[(i)] evaluate the three roots of $q_t$ along each specified path;
\item[(ii)] match only those three roots from one subdivision point to the next by minimum total displacement;
\item[(iii)] keep the branch-section points $0,t,1$ fixed by formula;
\item[(iv)] apply $w=1/(u-20i)$ and record every adjacent exchange in the real projection, with sign determined by \eqref{eq:crossing-sign};
\item[(v)] freely reduce the resulting Artin word.
\end{enumerate}
The reduced words are unchanged when each linear segment is subdivided into $300,700,1500$, or $3000$ pieces.  The tangency words reduce to the conjugate half twists displayed in Section~\ref{sec:ordered}, and the three $3+3$ words give the separating partitions recorded there; Appendix~\ref{app:star-infinity-check} checks the factor at infinity by an exact Artin-action calculation.

The continuation is numerical, not interval-certified.  Agreement under subdivision tests the stability of the Artin strings.  Once the strings are obtained, the braid and mapping-class-group calculations in Section~\ref{sec:ordered} and Appendix~\ref{app:star-infinity-check} are exact.  The implementation and numerical parameters will be included in the supplementary material.

\subsection{The star-system factor at infinity}\label{app:star-infinity-check}
For the factor at infinity, the word in Section~\ref{sec:ordered} can be checked directly.  It decomposes as
\[
 w_\infty=K_\infty m_\infty K_\infty^{-1},
\]
where
\begingroup\small
\[
\begin{aligned}
K_\infty={}&
\sigma_5\sigma_4\sigma_5\sigma_1^{-1}\sigma_4^{-1}\sigma_3^{-1}\sigma_2^{-1}
\sigma_5^{-1}\sigma_4^{-1}\sigma_3^{-1}\sigma_5^{-1}\sigma_4^{-1}\sigma_5^{-1}\\
&{}\cdot\sigma_3^{-1}\sigma_4^{-1}\sigma_5^{-1}\sigma_2^{-1}\sigma_3^{-1}\sigma_4^{-1},
\end{aligned}
\]
\endgroup
and
\[
 m_\infty=
 \sigma_3\sigma_2\sigma_3\sigma_3\sigma_2\sigma_1^{-1}
 \sigma_3\sigma_2^{-1}\sigma_3^{-1}\sigma_2\sigma_1\sigma_2
 \sigma_3\sigma_2\sigma_1\sigma_3\sigma_2\sigma_3.
\]
Apply the standard faithful Artin representation
\[
 B_6\longrightarrow\operatorname{Aut}(F_6),\qquad
 \sigma_i:x_i\mapsto x_ix_{i+1}x_i^{-1},\quad x_{i+1}\mapsto x_i,
\]
with all other free generators fixed.  Free reduction of the six images gives exactly the same automorphism for $m_\infty$ and $(\sigma_2\sigma_3)^6$.  Faithfulness therefore gives
\[
 m_\infty=(\sigma_2\sigma_3)^6\qquad\text{in }B_6.
\]
The permutation induced by $K_\infty$ is
\[
 1\mapsto2,\qquad2\mapsto6,\qquad3\mapsto3,\qquad
 4\mapsto1,\qquad5\mapsto5,\qquad6\mapsto4,
\]
so it carries the three-strand set $\{2,3,4\}$ to $\{1,3,6\}$.  Consequently the corresponding transported three-puncture circle separates
\[
 \{b_1,b_3,b_6\}\quad\text{from}\quad\{b_2,b_4,b_5\}.
\]
This proves the stated braid identity for the extracted word.

\subsection{Braid-level simplification of the seven star-system factors}\label{app:braid-simplification}
We record a direct braid-level reduction of the printed Artin strings to the seven Dehn-twist factors in Theorem~\ref{thm:orderedexplicit}.

For the four tangencies, exact reduction in $B_6$ gives
\[
 w_D=B_D\sigma_4B_D^{-1},\qquad
 w_R=B_R\sigma_3B_R^{-1},\qquad
 w_U=B_U\sigma_2B_U^{-1},\qquad
 w_L=B_L\sigma_3B_L^{-1},
\]
with the conjugating braids $B_D,B_R,B_U,B_L$ displayed in Section~\ref{sec:ordered}.  Applying $\Phi(\sigma_i)=T_i=t_{c_i}$ therefore gives
\[
 \Phi(w_D)=t_{a_D},\qquad \Phi(w_R)=t_{a_R},\qquad
 \Phi(w_U)=t_{a_U},\qquad \Phi(w_L)=t_{a_L}.
\]

The separating factors simplify similarly.  At $t=0$ the printed word is
\[
 w_0=(\sigma_2\sigma_1\sigma_2)^2(\sigma_5\sigma_4\sigma_5)^2.
\]
The braid relations give
\[
 (\sigma_2\sigma_1\sigma_2)^2=(\sigma_1\sigma_2)^3,
 \qquad
 (\sigma_5\sigma_4\sigma_5)^2=(\sigma_4\sigma_5)^3,
\]
and hence
\[
 w_0=\Omega:=(\sigma_1\sigma_2)^3(\sigma_4\sigma_5)^3.
\]
By Proposition~\ref{prop:deckcorrection-main},
\[
 \Phi(w_0)=\Phi(\Omega)=\iota t_{d_0}.
\]
Thus the geometric deck correction gives
\[
 \mathcal W_0=\iota\Phi(w_0)=t_{d_0}.
\]

At $t=1$, with $K_1=\sigma_5\sigma_4\sigma_3\sigma_2\sigma_1$, the printed word reduces exactly to
\[
 w_1=K_1\Omega K_1^{-1}.
\]
If $Q_1=\Phi(K_1)$ and $d_1=Q_1(d_0)$, then centrality of $\iota$ gives
\[
 \Phi(w_1)=Q_1(\iota t_{d_0})Q_1^{-1}=\iota t_{d_1},
 \qquad
 \mathcal W_1=\iota\Phi(w_1)=t_{d_1}.
\]

Finally Appendix~\ref{app:star-infinity-check} proves the exact identity
\[
 w_\infty=K_\infty(\sigma_2\sigma_3)^6K_\infty^{-1}.
\]
Writing $Q_\infty=\Phi(K_\infty)$ and letting $d'$ denote the separating lift of the boundary of the three-strand cluster in positions $2,3,4$, the two-chain relation gives
\[
 \Phi(w_\infty)=Q_\infty(T_2T_3)^6Q_\infty^{-1}
 =Q_\infty t_{d'}Q_\infty^{-1}=t_{d_\infty}.
\]
Thus no deck correction occurs for this particular representative at infinity.  Altogether the seven printed braid words lift, after the two geometric deck corrections at $1$ and $0$, to
\[
 \boxed{
 t_{a_D}\,t_{d_1}\,t_{a_R}\,t_{a_U}\,
 t_{d_\infty}\,t_{a_L}\,t_{d_0}},
\]
which is exactly the ordered word of Theorem~\ref{thm:orderedexplicit}.

\section{Alternative local braid coordinates}\label{app:independent-check}
For comparison, choose paths independently of the star system in Section~\ref{sec:ordered}.  Their local Artin coordinates exhibit the same lift ambiguity at the three $3+3$ degenerations.  Since these paths are only local choices, the order below is not a factorization order.

We use the same disk coordinate
\[
w=\frac{1}{u-20i}
\]
and the same crossing convention as in Section~\ref{sec:ordered}.  Put $\varepsilon=10^{-4}$.  The independent paths are
\begin{align*}
\gamma_L&:\ \frac14\to \frac9{50}+\frac{11}{50}i
 \to\tau_L+\frac3{25}i\to\tau_L+\varepsilon i,\\
\gamma_R&:\ \frac14\to \frac34-\frac{11}{50}i
 \to\frac{103}{100}-\frac{11}{50}i
 \to\tau_R-\frac3{25}i\to\tau_R-\varepsilon i,\\
\gamma_U&:\ \frac14\to \frac7{20}+\frac9{20}i
 \to\frac12+\frac95i\to\tau_U-\varepsilon i,\\
\gamma_D&:\ \frac14\to \frac7{20}-\frac9{20}i
 \to\frac12-\frac95i\to\tau_D+\varepsilon i,\\
\gamma_0&:\ \frac14\to\varepsilon,\\
\gamma_1&:\ \frac14\to\frac{21}{50}+\frac3{25}i
 \to\frac{29}{50}+\frac3{25}i
 \to\frac{17}{20}+\frac2{25}i\to1-\varepsilon,\\
\gamma_\infty&:\ \frac14\to\frac7{20}+\frac3{10}i
 \to\frac45+\frac12i\to2+\frac12i\to100.
\end{align*}

Their interiors avoid the critical values.  Continuing the three roots of $q_t$ along these paths, while keeping the branch sections $0,t,1$ labeled by their formulas, gives the following words by the numerical procedure of Appendix~\ref{app:braidcode}.

\begin{proposition}[Independent-path Artin coordinates]\label{prop:independent-braids}
The conjugating braids are
\begin{align*}
\beta_L&=\sigma_2\sigma_5\sigma_4\sigma_5\sigma_1\sigma_2,\\
\beta_R&=\sigma_5^{-1}\sigma_4^{-1}\sigma_3^{-1}\sigma_1\sigma_4\sigma_5\sigma_4,\\
\beta_U&=\sigma_5\sigma_4\sigma_1^{-1}\sigma_5\sigma_4^{-1}\sigma_3^{-1},\\
\beta_D&=\sigma_5^{-1}\sigma_4^{-1}\sigma_3^{-1}\sigma_5^{-1}
          \sigma_1\sigma_2^{-1}\sigma_3^{-1},\\
\kappa_0&=1,\qquad
\kappa_1=\sigma_5\sigma_4\sigma_3\sigma_2\sigma_1,\\
\kappa_\infty&=\sigma_5\sigma_4\sigma_3\sigma_4^{-1}\sigma_1^{-1}
\sigma_3^{-1}\sigma_5^{-1}\sigma_4^{-1}\sigma_5^{-1}.
\end{align*}
The four tangency braids are
\begin{equation}\label{eq:independent-tangencies}
\beta_L\sigma_3\beta_L^{-1},\qquad
\beta_R\sigma_3\beta_R^{-1},\qquad
\beta_U\sigma_2\beta_U^{-1},\qquad
\beta_D\sigma_4\beta_D^{-1}.
\end{equation}
If
\[
\Omega=(\sigma_1\sigma_2)^3(\sigma_4\sigma_5)^3,
\]
the three $3+3$ disk-braid representatives are
\begin{equation}\label{eq:independent-separating}
\Omega,\qquad
\kappa_1\Omega\kappa_1^{-1},\qquad
\kappa_\infty\Omega\kappa_\infty^{-1}.
\end{equation}
\end{proposition}

The same words are obtained with $300$, $700$, $1500$, and $3000$ subdivisions of each linear segment.  Closing the four paths by small positive circles about the tangency values gives the conjugates in \eqref{eq:independent-tangencies}, while the calculation at $0,1,\infty$ gives \eqref{eq:independent-separating}.  This subdivision test is numerical and does not provide interval certification.

\begin{figure}[ht]
\centering
\begin{tikzpicture}[scale=1.02]
  \coordinate (b1) at (0,0);
  \coordinate (b2) at (1.25,0);
  \coordinate (b3) at (2.50,0);
  \coordinate (b4) at (3.75,0);
  \coordinate (b5) at (5.00,0);
  \coordinate (b6) at (6.25,0);

  \draw[magenta!75!black,thick,rounded corners=5pt]
    (-0.42,-0.62) rectangle (2.93,0.96);
  \node[magenta!75!black] at (1.25,1.24) {$\Delta$};

  \draw[blue!80!black,thick]
    (b1) .. controls (0.32,0.36) and (0.93,0.36) .. (b2);
  \draw[red!75!black,thick]
    (b2) .. controls (1.57,0.36) and (2.18,0.36) .. (b3);
  \draw[green!55!black,thick]
    (b3) .. controls (2.82,0.36) and (3.43,0.36) .. (b4);
  \draw[orange!90!black,thick]
    (b4) .. controls (4.07,0.36) and (4.68,0.36) .. (b5);
  \draw[magenta!80!black,thick]
    (b5) .. controls (5.32,0.36) and (5.93,0.36) .. (b6);

  \node[blue!80!black] at (0.625,0.53) {$\eta_1$};
  \node[red!75!black] at (1.875,0.53) {$\eta_2$};
  \node[green!55!black] at (3.125,0.53) {$\eta_3$};
  \node[orange!90!black] at (4.375,0.53) {$\eta_4$};
  \node[magenta!80!black] at (5.625,0.53) {$\eta_5$};

  \foreach \p/\lab in {b1/$b_1$,b2/$b_2$,b3/$b_3$,b4/$b_4$,b5/$b_5$,b6/$b_6$}{
    \fill (\p) circle (1.5pt);
    \node[below=4pt] at (\p) {\lab};
  }
\end{tikzpicture}
\caption{The standard Artin arcs on the reference branch sphere.  For each $i$, the curve
$c_i=p^{-1}(\eta_i)$ is the lift of the arc joining $b_i$ to $b_{i+1}$, and
$T_i=t_{c_i}$ is the positive Dehn twist about that curve.  The curves
$c_1,\ldots,c_5$ form the standard five-chain on $\Sigma_2$.  The lift of
the three-puncture circle $\Delta$ is the separating curve $d$ used below.}
\label{fig:independent-standard-arcs}
\end{figure}
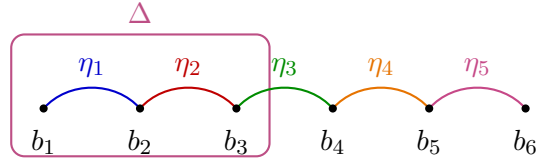

We next compare the lifts.  Let $c_1,\ldots,c_5$ be the standard five-chain on the reference genus-two fiber and write $T_i=t_{c_i}$.  Under the Artin-chain homomorphism $\Phi(\sigma_i)=T_i$, Proposition~\ref{prop:deckcorrection-main} gives
\begin{equation}\label{eq:independent-lift}
\Phi(\Omega)=(T_1T_2)^3(T_4T_5)^3=\iota\,t_d,
\qquad
t_d=(T_1T_2)^6=(T_4T_5)^6.
\end{equation}
Thus the raw disk-braid lift is the other preimage of the spherical class from the local Lefschetz monodromy $t_d$.  If $Q_j=\Phi(\kappa_j)$ and
\[
d_0^{\mathrm{ind}}=d,\qquad
d_1^{\mathrm{ind}}=Q_1(d),\qquad
d_\infty^{\mathrm{ind}}=Q_\infty(d),
\]
then the geometric separating factors are
\[
t_{d_0^{\mathrm{ind}}},\qquad
t_{d_1^{\mathrm{ind}}}=Q_1t_dQ_1^{-1},\qquad
t_{d_\infty^{\mathrm{ind}}}=Q_\infty t_dQ_\infty^{-1},
\]
whereas the three raw lifts of \eqref{eq:independent-separating} are obtained by multiplying these factors by $\iota$.  Since $\iota$ is central, the possible choice of lift of $Q_j$ does not affect the transported separating curve.

There is also a short homological check.  For the symplectic basis $(a_1,b_1,a_2,b_2)$ used in Section~\ref{sec:ordered}, take
\[
[c_1]=a_1,\qquad [c_2]=b_1,\qquad [c_3]=a_1+a_2,\qquad
[c_4]=b_2,\qquad [c_5]=a_2.
\]
The four transported nonseparating cycles determined by the independent paths have classes
\begin{equation}\label{eq:independent-H1}
\begin{aligned}
\relax[a_L^{\mathrm{ind}}]&=(0,-1,0,-1)^T,&
[a_R^{\mathrm{ind}}]&=(2,0,1,0)^T,\\
[a_U^{\mathrm{ind}}]&=(-2,1,-1,1)^T,&
[a_D^{\mathrm{ind}}]&=(-2,-1,-1,-1)^T.
\end{aligned}
\end{equation}
Each vector is primitive, and the corresponding conjugate in \eqref{eq:independent-tangencies} acts by the symplectic transvection about that vector.  The three factors in \eqref{eq:independent-lift} obtained from $t_d$ and its conjugates act trivially on $H_1$.  Thus this second path system independently gives four nonseparating and three separating local factors, with the same lift correction at the $3+3$ values as in Proposition~\ref{prop:deckcorrection-main}.

The global product remains the one associated with the distinguished star system of Theorem~\ref{thm:orderedexplicit}.

\section*{Acknowledgments}
The author thanks Fabrizio Catanese for several helpful conversations about Xiao Gang's genus-two fibrations, Ludmil Katzarkov for discussions of Moishezon's braid-monodromy techniques, and Denis Auroux for a brief discussion of Xiao's fibration some years ago.  Most of the work underlying this paper dates back several years; the project was taken up again and completed more recently.  The author gratefully acknowledges the hospitality of the University of Bayreuth, where he was a Humboldt Research Fellow during the summers of 2021 and 2022 and completed part of this work.  An LLM-based tool was used for limited assistance with computational preparation, grammar checking, and figure generation.  The author checked the mathematical arguments and takes full responsibility for the contents of the paper.


\begin{thebibliography}{99}
\bibitem{AM}
A.~Akhmedov and N.~Monden,
\href{https://doi.org/10.1215/21562261-2019-0067}{\emph{Genus 2 Lefschetz fibrations with $b_2^+=1$ and $c_1^2=1,2$}},
\emph{Kyoto J. Math.} \textbf{60} (2020), no.~4, 1419--1451,
DOI: \href{https://doi.org/10.1215/21562261-2019-0067}{10.1215/21562261-2019-0067};
\href{https://arxiv.org/abs/1509.01853}{arXiv:1509.01853} [math.GT], first posted September~6, 2015.

\bibitem{AKBranched}
D.~Auroux and L.~Katzarkov,
\href{https://doi.org/10.1007/PL00005795}{\emph{Branched coverings of $\mathbb{CP}^2$ and invariants of symplectic $4$-manifolds}},
\emph{Invent. Math.} \textbf{142} (2000), no.~3, 631--673,
DOI: \href{https://doi.org/10.1007/PL00005795}{10.1007/PL00005795}.

\bibitem{AKDoubling}
D.~Auroux and L.~Katzarkov,
\href{https://doi.org/10.4310/PAMQ.2008.v4.n2.a2}{\emph{A degree doubling formula for braid monodromies and Lefschetz pencils}},
\emph{Pure Appl. Math. Q.} \textbf{4} (2008), no.~2, 237--318,
DOI: \href{https://doi.org/10.4310/PAMQ.2008.v4.n2.a2}{10.4310/PAMQ.2008.v4.n2.a2};
\href{https://arxiv.org/abs/math/0605001}{arXiv:math/0605001}.

\bibitem{BHPV}
W.~P.~Barth, K.~Hulek, C.~A.~M.~Peters, and A.~Van de Ven,
\href{https://doi.org/10.1007/978-3-642-57739-0}{\emph{Compact complex surfaces}}, 2nd ed., Ergeb. Math. Grenzgeb. (3), vol.~4,
Springer-Verlag, Berlin, 2004,
DOI: \href{https://doi.org/10.1007/978-3-642-57739-0}{10.1007/978-3-642-57739-0}.

\bibitem{BH}
J.~S.~Birman and H.~M.~Hilden,
\href{https://doi.org/10.1515/9781400822492-007}{\emph{On the mapping class groups of closed surfaces as covering spaces}},
in \emph{Advances in the theory of Riemann surfaces}, Ann. of Math. Stud., no.~66,
Princeton Univ. Press, Princeton, NJ, 1971, pp.~81--115,
DOI: \href{https://doi.org/10.1515/9781400822492-007}{10.1515/9781400822492-007}.

\bibitem{Endo}
H.~Endo,
\href{https://doi.org/10.1007/s002080050012}{\emph{Meyer's signature cocycle and hyperelliptic fibrations}},
Math. Ann. \textbf{316} (2000), no.~2, 237--257,
DOI: \href{https://doi.org/10.1007/s002080050012}{10.1007/s002080050012}.

\bibitem{FM}
B.~Farb and D.~Margalit,
\href{https://doi.org/10.23943/princeton/9780691147949.001.0001}{\emph{A primer on mapping class groups}}, Princeton Math. Ser., vol.~49,
Princeton Univ. Press, Princeton, NJ, 2012,
DOI: \href{https://doi.org/10.23943/princeton/9780691147949.001.0001}{10.23943/princeton/9780691147949.001.0001}.

\bibitem{KatzMonodromy}
L.~Katzarkov,
\href{https://doi.org/10.1023/A:1022375625839}{\emph{Monodromy invariants: from symplectic to smooth manifolds}},
\emph{Acta Appl. Math.} \textbf{75} (2003), no.~1--3, 85--103,
DOI: \href{https://doi.org/10.1023/A:1022375625839}{10.1023/A:1022375625839}.

\bibitem{MoishezonStable}
B.~Moishezon,
\href{https://doi.org/10.1007/BFb0090891}{\emph{Stable branch curves and braid monodromies}},
in \emph{Algebraic Geometry (Chicago, 1980)}, Lecture Notes in Math., vol.~862,
Springer, Berlin--New York, 1981, pp.~107--192,
DOI: \href{https://doi.org/10.1007/BFb0090891}{10.1007/BFb0090891}.

\bibitem{MoishezonBraidsII}
B.~Moishezon,
\href{https://www.ams.org/books/conm/044/}{\emph{Algebraic surfaces and the arithmetic of braids. II}},
in \emph{Combinatorial Methods in Topology and Algebraic Geometry (Rochester, N.Y., 1982)},
Contemp. Math., vol.~44, Amer. Math. Soc., Providence, RI, 1985, pp.~311--344.

\bibitem{MT1}
B.~Moishezon and M.~Teicher,
\href{https://doi.org/10.1090/conm/078/975093}{\emph{Braid group technique in complex geometry. I. Line arrangements in $\mathbb{CP}^2$}},
in \emph{Braids (Santa Cruz, CA, 1986)}, Contemp. Math., vol.~78,
Amer. Math. Soc., Providence, RI, 1988, pp.~425--555,
DOI: \href{https://doi.org/10.1090/conm/078/975093}{10.1090/conm/078/975093}.

\bibitem{Pardini}
R.~Pardini,
\href{https://doi.org/10.1515/crll.1991.417.191}{\emph{Abelian covers of algebraic varieties}},
J. Reine Angew. Math. \textbf{417} (1991), 191--214,
DOI: \href{https://doi.org/10.1515/crll.1991.417.191}{10.1515/crll.1991.417.191}.

\bibitem{Xiao}
G.~Xiao,
\href{https://doi.org/10.1007/BFb0075351}{\emph{Surfaces fibr\'ees en courbes de genre deux}},
Lecture Notes in Math., vol.~1137, Springer-Verlag, Berlin, 1985 (in French),
DOI: \href{https://doi.org/10.1007/BFb0075351}{10.1007/BFb0075351}.
\end{thebibliography}
\end{document}